\documentclass[11pt,reqno]{amsart}
\usepackage[T1]{fontenc}
\usepackage{amsfonts}
\usepackage{amssymb}
\usepackage{mathrsfs}
\usepackage[latin1]{inputenc}
\usepackage{amsmath}
\usepackage[english]{babel}

\newtheorem{theorem}{Theorem}[section]
\newtheorem{lemma}[theorem]{Lemma}

\newtheorem{remark}[theorem]{Remark}
\newtheorem{prop}[theorem]{Proposition}

\newtheorem{corollary}[theorem]{Corollary}

\numberwithin{equation}{section}

\newcommand{\R}{{\mathbb R}}
\newcommand{\D}{{\mathbb D}}

\newcommand{\C}{{\mathbb C}}
\newcommand{\N}{{\mathbb N}}

\newcommand{\cL}{{\mathcal L}}

\newcommand{\ve}{\varepsilon}

\newcommand{\su}{\subseteq}

\newcommand{\sC}{\mathsf{C}}
\newcommand{\ov}{\overline}

\newcommand\Ker{\mathop{\rm Ker}}

\begin{document}

\title{The Ces\`aro operator on power series spaces }

\author{Angela\,A. Albanese, Jos\'e Bonet, Werner \,J. Ricker }

\dedicatory{Dedicated to the late Pawe{\l} Doma\'nski}

%\thanks{\textsuperscript{*} Corresponding author}
\thanks{This article is accepted for publication in Studia Mathematica}

\address{ Angela A. Albanese\\
Dipartimento di Matematica e Fisica ``E. De Giorgi''\\
Universit\`a del Salento- C.P.193\\
I-73100 Lecce, Italy}
\email{angela.albanese@unisalento.it}

\address{Jos\'e Bonet \\
Instituto Universitario de Matem\'{a}tica Pura y Aplicada
IUMPA \\
Universitat Polit\`ecnica de
Val\`encia\\
E-46071 Valencia, Spain} \email{jbonet@mat.upv.es}

\address{Werner J.  Ricker \\
Math.-Geogr. Fakult\"{a}t \\
 Katholische Universit\"{a}t
Eichst\"att-Ingol\-stadt \\
D-85072 Eichst\"att, Germany}
\email{werner.ricker@ku.de}
\markboth{A.\,A. Albanese, J. Bonet, W. \,J. Ricker}%
{\MakeUppercase{ }}

\begin{abstract} The discrete Ces\`aro operator $\sC$ is investigated in the class of power series spaces $\Lambda_0(\alpha)$ of finite type. Of main interest is its spectrum, which is distinctly different when the underlying Fr\'echet space $\Lambda_0(\alpha)$ is nuclear as for the case when it is not. Actually, the nuclearity of $\Lambda_0(\alpha)$ is characterized via certain properties of the spectrum of $\sC$. Moreover, $\sC$ is always power bounded, uniformly mean ergodic  and, whenever $\Lambda_0(\alpha)$ is nuclear, also has the property that the range $(I-\sC)^m(\Lambda_0(\alpha))$ is closed in $\Lambda_0(\alpha)$, for each $m\in\N$.
\end{abstract}

\maketitle

%\newpage
\markboth{A.\,A. Albanese, J. Bonet and W.\,J. Ricker}%
{\MakeUppercase{Spectrum and compactness of the Ces\`aro operator }}

\section{Introduction.}

The Ces\`aro operator $\sC$, and some of its generalizations,  have been investigated in many Banach sequence spaces and Banach  spaces of analytic functions. Some of these generalizations and certain unifying approaches to them  can be found in  \cite{AP, GS, Rud} and the references therein. The situation when $\sC$ acts in a Fr\'echet
(locally convex) space is also of interest. The setting of this paper is the discrete Ces\`aro operator $\sC$  defined on the linear space $\C^\N$ (consisting of all scalar sequences) by
$$
\sC x:=\left(x_1, \frac{x_1+x_2}{2}, \ldots, \frac{x_1+\ldots +x_n}{n}, \ldots\right),\quad x=(x_n)_{n\in\N}\in \C^\N.
$$
The linear operator $\sC$ is said to \textit{act} in a vector subspace  $X\su\C^\N$ if it maps $X$ into itself. Two fundamental questions are: Is $\sC\colon X\to X$ continuous and, if so, what is
its  spectrum? Amongst the classical Banach spaces $X\su \C^\N$ where precise answers are known we mention $\ell_p$ ($1<p<\infty$), \cite{BHS}, \cite{L}, and $c_0$, \cite{L}, \cite{R},  both $c$, $\ell_\infty$, \cite{A-B}, \cite{L}, as well as $ces_p$, $p\in \{0\}\cup (1,\infty)$, \cite{C-R}, the Bachelis spaces $N^p$, $2\leq p<\infty$, \cite{C-R1}, the spaces of bounded variation $bv_0$, \cite{Ok}, and bounded $p$-variation $bv_p$, $1\leq p<\infty$, \cite{A-B2}, and the  weighted Banach spaces $\ell_p(w)$, $1<p<\infty$, \cite{ABR}, and $c_0(w)$, \cite{ABR-9}. The behaviour of $\sC$  (and other Hausdorff operators) in the Fr\'echet space $\C^\N$, and of its dual operator $\sC'$,  are known since the work of Hausdorff \cite{Hau}; see also \cite{H43, H49, Rud}. The discrete Ces\`aro operator $\sC$ has also been studied  in the \textit{Fr\'echet} spaces $\ell^{p+} := \cap_{q>p} \ell_q$,    \cite{ABR-1}.  There is no claim that this list of spaces (and references) is complete.

The aim of this paper is to investigate the behaviour of $\sC$ in the class of Fr\'echet sequence spaces consisting of the  \textit{power series spaces of finite type} $\Lambda_0(\alpha)\su\C^\N$, where $\alpha=(\alpha_n)_n$ is any positive sequence satisfying $\alpha_n\uparrow\infty$ (see Section 2 for the definition). Such spaces play an important role in the structure theory of Fr\'echet spaces, \cite{MV}, \cite{V1}, \cite{V2}.
First, $\sC\colon  \Lambda_0(\alpha)\to \Lambda_0(\alpha)$ is always  continuous (see Proposition \ref{cont}), which is not necessarily the case for power series spaces of infinite order. In Section 2 a detailed investigation is made of the spectrum  $\sigma(\sC;\Lambda_0(\alpha))$ and the point spectrum $\sigma_{pt}(\sC;\Lambda_0(\alpha))$ of $\sC\colon  \Lambda_0(\alpha)\to \Lambda_0(\alpha)$. A remarkable feature arises in this regard. It is known that   $\Lambda_0(\alpha)$ is always a Fr\'echet  Schwartz space but that it is \textit{nuclear} if and only if $\lim_{n\to\infty}\frac{\log n}{\alpha_n}=0$. These facts are totally independent of the Ces\`aro operator $\sC$. Nevertheless, certain spectral properties of $\sC$ turn out to  \textit{characterize} the nuclearity of $\Lambda_0(\alpha)$. Indeed, with the notation $\Sigma:=\{\frac{1}{m}\colon m\in\N\}$ and $\Sigma_0:=\{0\}\cup\Sigma$, the equivalence of the following assertions is established (see Propositions \ref{P-K1} and \ref{P-K2} and Corollary \ref{C.Su-1}), where we recall that $\sC\colon \C^\N\to\C^\N$ always possesses an inverse operator (denoted by $\sC^{-1}$).
\begin{itemize}
\item[\rm (i)] $\Lambda_0(\alpha)$ is \textit{nuclear}.
\item[\rm (ii)] $\sC^{-1}\colon \Lambda_0(\alpha)\to \Lambda_0(\alpha)$ is \textit{continuous}, i.e., $0\not\in\sigma(\sC;\Lambda_0(\alpha))$.
\item[\rm (iii)] $\sigma_{pt}(\sC;\Lambda_0(\alpha))=\Sigma$.
\item[\rm (iv)] $\sigma_{pt}(\sC;\Lambda_0(\alpha))\setminus\{1\}\not=\emptyset$.
\item[\rm (v)] $\sigma(\sC;\Lambda_0(\alpha))=\Sigma$.
\item[\rm (vi)] $\sigma_{pt}(\sC;\Lambda_0(\alpha))=\sigma(\sC;\Lambda_0(\alpha))$.
\end{itemize}

Remark \ref{R.S-C} shows  that  always $\Sigma\su \sigma(\sC;\Lambda_0(\alpha))\su \left\{\lambda\in\C\colon \left|\lambda-\frac{1}{2}\right|\leq \frac{1}{2}\right\}$.
On the other hand, if there exists a real number $s\geq 1$ satisfying $\sum_{n=1}^\infty\frac{e^{\alpha_n}}{n^s}<\infty$, then $\Lambda_0(\alpha)$ fails to be nuclear and the inclusions
\[
\{1\}\cup \left\{\lambda\in\C\colon \left|\lambda-\frac{1}{2}\right|< \frac{1}{2}\right\}\su \sigma(\sC;\Lambda_0(\alpha))\su \left\{\lambda\in\C\colon \left|\lambda-\frac{1}{2}\right|\leq \frac{1}{2}\right\}
\]
hold; see Proposition \ref{P-spectrum-no}.  For $\alpha_n:=n$, for $n\in\N$, the space $\Lambda_0(\alpha)$ is isomorphic to the nuclear Fr\'echet space $H(\D)$ of all analytic functions on the  open unit disc $\D$ (with the topology of uniform convergence on compact subsets of $\D$).  Our results imply, via different methods, the known fact  that $\sigma(\sC;H(\D))=\sigma_{pt}(\sC; H(\D))=\Sigma$; see  \cite[pp.65-68]{Dah} and also \cite{AP}.

Section 3 is devoted to  mean ergodic properties of $\sC\colon \Lambda_0(\alpha)\to \Lambda_0(\alpha)$. In Proposition \ref{p.unif}  it is shown that $\sC$ is always power bounded and uniformly mean ergodic. If $\Lambda_0(\alpha)$ is nuclear,  then the range $(I-\sC)^m(\Lambda_0(\alpha))$ is always  closed in $\Lambda_0(\alpha)$ for every $m\in\N$ (see Proposition \ref{p.range}).

\section{Continuity and spectrum of $\sC$  on $\Lambda_0(\alpha)$}

Let $X$ be a locally convex Hausdorff space (briefly, lcHs) and $\Gamma_X$ a system of continuous seminorms determining the topology of $X$.  The identity operator on $X$ is denoted by $I$ and  $\cL(X)$ denotes the space of all continuous linear operators from $X$ into itself. Let $\cL_s(X)$ denote $\cL(X)$ endowed with the strong operator topology $\tau_s$ which is determined by the  seminorms $T \rightarrow q_x(T):=q(Tx)$, for  each $x\in X$ and $q\in \Gamma_X$. Moreover,  $\cL_b(X)$ denotes $\cL(X)$ equipped with the  topology $\tau_b$ of uniform convergence on bounded subsets of $X$ which is determined by the  seminorms $T \rightarrow q_B(T):=\sup_{x\in B}q(Tx)$, for  each $B \subseteq X$ bounded and $q\in \Gamma_X$.

A sequence  $A = (a_k)_k$ of functions $a_k: \N \to [0, \infty)$  is called a \textit{K\"{o}the matrix} on $\N$ if
 $0 \leq a_k (n) \le a_{k+1} (n)$ for all $n \in \N$ and $k\in\N$, and if for each $n \in \N$ there is $k \in \N$ with $a_k(n)>0$. The \textit{K\"{o}the echelon space of order $0$} associated to  $A$ is
\[
\lambda_0 (A) : = \{x \in \C^{\N} :  \lim_{n  }a_k (n) x_n = 0 , \ \forall k \in\N\},
\]
which is a Fr\'echet space relative to the increasing sequence of canonical seminorms
$$
q^{(\infty)}_k (x) := \sup_{n \in \N } a_k (n ) |x_n |, \quad  x\in \lambda_0 (A), \quad k \in \N.
$$
Then $\lambda_0(A)=\cap_{k \in \N} c_0(a_k)$, with  $c_0(a_k)$ the usual weighted Banach space. The space $\lambda_0 (A)$ is given the projective limit topology, i.e., $\lambda_0(A)= {\rm proj}_{k \in \N} c_0(a_k)$. For the  theory of the K\"othe echelon spaces $\lambda_p (A)$, $1 \leq p \leq \infty$ or $p=0$,  see \cite{MV}.

Fix a sequence $\{r_k\}_{k\in\N}\su (0,1)$ satisfying $\lim_{k\to\infty}r_k=1$ with $r_k<r_{k+1}$ for all $k\in\N$. Moreover, let $\alpha:=\{\alpha_n\}_{n\in\N}\su (1,\infty)$ satisfy $\lim_{n\to\infty}\alpha_n=\infty$ with $\alpha_n<\alpha_{n+1}$ for all $n\in\N$; we simply write $\alpha_n\uparrow\infty$. For each $ k\in\N$ define $w_k\colon \N\to (0,\infty)$ by $w_k(n):=(r_k)^{\alpha_n}$, for $n\in\N$, in which case $A=(w_k)_{k}$ is a K\"othe matrix. The \textit{power series space of finite type} associated to $\alpha$ is defined by
\[
\Lambda_0(\alpha):=\left\{x\in\C^\N\colon \lim_{n\to\infty}w_k(n)|x_n|=0,\ \forall k\in\N\right\},
\]
and coincides with $\lambda_0(A)$ in the above notation. Then $\Lambda_0(\alpha)=\cap_{k\in\N}c_0(w_k)$ and its Fr\'echet space lc-topology is generated by the increasing sequence of norms
\begin{equation}\label{eq.seminorme}
p_k(x):=\sup_{n\in\N}w_k(n)|x_n|, \quad x\in \Lambda_0(\alpha), \quad k \in \N.
\end{equation}

\begin{remark}\label{R.W1}\rm
(i) The space $\Lambda_0(\alpha)$ and its topology are independent of the  increasing sequence $\{r_k\}_{k\in\N}$ tending to $1$.
 We always choose  $r_k=e^{-1/k}$ so that $w_k(n)=e^{-\alpha_n/k}$, for all $k, n \in \N$. Then $w_k \le w_l$ on $\N$ whenever $k \le l$.

(ii) For each $k\in\N$, the condition $\frac{r_k}{r_{k+1}}<1$ implies that $\left(\frac{w_k(n)}{w_{k+1}(n)}\right)_n\in c_0$ and so,  for every $\alpha$, the space $\Lambda_0(\alpha)$ is Fr\'echet Schwartz, \cite[Theorem 27.9, Proposition 27.10]{MV}, and hence, also   Fr\'echet Montel.

(iii) Since $\Lambda_0(\alpha)$ coincides with $\lambda_0(A)$ in the above notation, for the matrix $A=(w_k)_k$, Proposition 28.16  in \cite{MV} yields that $\Lambda_0(\alpha)$ is nuclear if and only if $\Lambda_0(\alpha)=\lambda^2(A)=\cap_{k\in\N}\ell_2(w_k)$, \cite[Definition, p.326]{MV}. According to \cite[Proposition 29.6]{MV} we can conclude that $\Lambda_0(\alpha)$ is a nuclear Fr\'echet space if and only if
$\lim_{n\to\infty}\frac{\log n}{\alpha_n}=0$. Observe that power series spaces are defined in Chapter 29 of \cite{MV} using $\ell_2$-norms. Examples of sequences $\alpha_n\uparrow \infty$ such that $\Lambda_0(\alpha)$ is \textit{not} nuclear include $\alpha_n:=\beta \log n$, for $n\in\N$ and any fixed $\beta>0$, and $\alpha_n:=\log (\log n)$ for $n>e^e$.
\end{remark}

The nuclearity criterion $\lim_{n\to\infty}\frac{\log n}{\alpha_n}=0$ will play a significant role later. Recall that $x=(x_n)_n\in\C^\N$ belongs to the space $s$ of \textit{rapidly decreasing sequences} if and only if $(n^k x_n)_n$ is a bounded sequence for each $k \in \N$.

\begin{lemma}\label{nuclcond}
Let $0<r<1$ and let the sequence $\alpha$ satisfy $\alpha_n\uparrow\infty$. The strictly decreasing sequence $w=(w(n))_n$, with $w(n):= r^{\alpha_n}$, for $n \in \N,$ satisfies $w\in c_0$. Moreover, $w$ belongs to $s$ if and only if $\lim_{n\to\infty}\frac{\log n}{\alpha_n}=0$.
\end{lemma}

\begin{proof}
Set $a:=(1/r)>1$. Suppose $\lim_{n\to\infty}\frac{\log n}{\alpha_n}=0$. For fixed $k \in \N$ we have
$$
n^k r^{\alpha_n} = \exp\left( \alpha_n\left(k \frac{\log n}{\alpha_n} - \log a \right)\right),\quad n\in\N,
$$
which  tends to $0$ for $n\to\infty$  since $\log (a)>0$ with $\frac{\log n}{\alpha_n} \rightarrow 0$ and $\alpha_n \uparrow \infty$. In particular, $(n^k r^{\alpha_n})_n\in c_0\su\ell^\infty$. Since $k\in\N$ is arbitrary, $w\in s$.

Assume that $w \in s$. For each $k \in \N$, it follows from the definition of $w$ that
$\lim_{n \rightarrow \infty} (\alpha_n (\log a) - k \log n) = \infty$. So, there is $n(k) \in \N$ such that $(\alpha_n (\log a) - k \log n) > k$ for $n \geq n(k)$. Given $M>0$, select $k \in \N$ with $k>M(\log a)$. Then, for $n \geq n(k)$, we have  $\alpha_n/\log n > M$. Thus, $\lim_{n\to\infty}\frac{\log n}{\alpha_n}=0$.
\end{proof}

\begin{remark}\label{R.SW}\rm Let $\alpha$ be any sequence satisfying $\alpha_n \uparrow \infty$. Define
\begin{equation}\label{e.valpha}
v(\alpha):=\inf\{\alpha_{n+1}-\alpha_{n}\colon n\in\N\}.
\end{equation}
If $v(\alpha)>0$, then \eqref{e.valpha} implies $\lim_{n\to\infty}\frac{\log n}{\alpha_n}=0$ by the Stolz-Ces\`aro criterion, \cite[Ch.3, Theorem 1.22]{M}.

(a) Whenever  $\lim_{n\to\infty}(\alpha_{n+1}-\alpha_{n})$ exists in $(0,\infty]$, then necessarily $v(\alpha)>0$.

(b) For the sequence $\alpha$ given by $\alpha_1=2$ and, for each $k\in\N$,  by $\alpha_{2k}=3k$ and $\alpha_{2k+1}=2+\alpha_{2k}$, we have $\alpha_n \uparrow \infty$ with $(\alpha_{n+1}-\alpha_{n})\in \{1,\ 2\}$, for all $n\in\N$. Note that the sequence $(\alpha_{n+1}-\alpha_{n})_n$ is \textit{not} convergent. On the other hand, $v(\alpha)=1$.

(c) Set $\alpha_n:=\sqrt{n}$, for $n\in\N$. Then  $\alpha_n \uparrow \infty$ with $(\alpha_{n+1}-\alpha_{n})=\frac{1}{\sqrt{n+1}+\sqrt{n}}\to 0$ for $n\to\infty$, i.e., $v(\alpha)=0$. Nevertheless, it is still the case that $\lim_{n\to\infty}\frac{\log n}{\alpha_n}=0$.

(d) Let $\alpha_n:=\beta\log (n+1)$, for $n\in\N$, with $\beta>0$ fixed. Then $\alpha_n \uparrow \infty$ with $(\alpha_{n+1}-\alpha_{n})=\beta \log\left(1+\frac{1}{n+1}\right)\to 0$ for $n\to\infty$, i.e., $v(\alpha)=0$. In this case, $\lim_{n\to\infty}\frac{\log n}{\alpha_n}=\frac{1}{\beta}>0$. In particular, $\Lambda_0(\alpha)$ is not nuclear.

\end{remark}

\begin{prop}\label{cont} Let $\alpha$ be any sequence with $\alpha_n\uparrow \infty$.
The Ces\`aro operator $\sC$ acts continuously on $\Lambda_0(\alpha)$ and satisfies
\begin{equation}\label{e.pb}
p_k(\sC x)\leq p_k(x),\quad  x\in \Lambda_0(\alpha),
\end{equation}
for each $k\in\N$, with $\{p_k\}_{k\in\N}$  being the norms  in \eqref{eq.seminorme}.
\end{prop}

\begin{proof}
Since $w_k=(w_k(n))_n$ is decreasing,  Corollary 2.3(i) in \cite{ABR-9} implies that
$\sC\in \cL(c_0(w_k))$ and $p_k(\sC x)\leq p_k(x)$, for $x\in c_0(w_k)$ and $k\in\N$.
\end{proof}

Let $A=(a_k)_k$ be a K\"othe matrix. Since $\lambda_0(A)={\rm proj}_{k} c_0(a_k)$ and $\lambda_0(A)$ is dense in $c_0(a_k)$ for each $k\in\N$, the Ces\`aro operator $\sC$ acts continuously on  $\lambda_0(A)$ if and only for each $k\in\N$ there exists $l>k$ such that $\sC\colon c_0(a_l)\to c_0(a_k)$, acting between Banach spaces, is continuous. Applying \cite[Theorem 4.51-C]{T} and proceeding as in the proof of Proposition 2.2(i) in \cite{ABR-9}, this turns out to be equivalent to the fact that for each $k\in\N$ there exists $l>k$ such that
\[
\sup_n \frac{a_k(n)}{n} \sum_{m=1}^n \frac{1}{a_l(m)} < \infty.
\]
If we take, for example, $a_k(n)=n^k$ (in which case $\lambda_0(A)=s$) or $a_k(n)=k^n$ for all $n,k \in \N$, then the sequence $\left(\frac{a_k(n)}{n} \sum_{m=1}^n \frac{1}{a_l(n)}\right)_n$ is unbounded and so $\sC \notin \cL(\lambda_0(A))$ for these K\"othe matrices $A$. This is why we restrict our attention to the operator $\sC$ when acting on power series spaces $\Lambda_0(\alpha)$ of finite type.

Since $\Lambda_0(\alpha)$ is Montel and $\sC\in \cL(\Lambda_0(\alpha))$, it follows that $\sC$ is always a \textit{Montel operator} (i.e., maps bounded sets to relatively compact sets). Recall that  an operator $T \in \cL(X)$, with $X$ a Fr\'echet space, is \textit{compact} if there exists a neighbourhood $U$ of $0$ such that $T(U)$ is relatively compact in $X$.

\begin{prop}\label{comp} For every sequence $\alpha$ satisfying $\alpha_n\uparrow\infty$ the corresponding Ces\`aro operator $\sC\in\cL(\Lambda_0(\alpha))$ fails to be compact.
\end{prop}

\begin{proof}
Suppose $\sC$ is compact. Then there is $k\in\N$ such that
$\sC\colon c_0(w_k) \rightarrow c_0(w_l)$   is compact for all $l >k$, as a linear map between Banach spaces. In particular, $\sC\colon c_0(w_k) \rightarrow c_0(w_{k+1})$ is compact. Since $c_0(w_{k+1})\su c_0(w_{k})$ continuously (via the identity map), it follows $\sC\in \cL(c_0(w_k))$ is compact. By Proposition 3.9 of \cite{ABR-9}  the weight $w_k\in s$. Hence, via Lemma \ref{nuclcond} with $r=r_k$, we have $\lim_{n\to\infty}\frac{\log n}{\alpha_n}=0$, i.e.,  $\Lambda_0(\alpha)$ is nuclear.

Since $\sC\colon c_0(w_k) \rightarrow c_0(w_{2k})$   is compact,
Proposition 2.2 of \cite{ABR-9} yields that
\begin{equation}\label{eq.comp-1}
\sup_{n\in\N} A_k(n):=\sup_{n\in\N}\frac{w_{2k}(n)}{n} \sum_{m=1}^n \frac{1}{w_k(m)}<\infty.
\end{equation}
But, for each $n\in\N$ we have
\[
A_k(n)=\frac{\exp(-\alpha_n/2k)}{n}\sum_{m=1}^n e^{\alpha_m/k}\geq \frac{1}{n}\exp(-\alpha_n/2k)\exp(\alpha_n/k)=e^{p\alpha_n-\log n}
\]
with $p=\frac{1}{2k}$. Since $\lim_{n\to\infty}\frac{\log n}{\alpha_n}=0$, there exists $N\in\N$ such that $(p\alpha_n-\log n)\geq \frac{p\alpha_n}{2}$ for all $n\geq N$. Accordingly,
$A_k(n)\geq e^{p\alpha_n/2}$, for $ n \geq N$,
which contradicts \eqref{eq.comp-1}. So, $\sC\in \cL(\Lambda_0(\alpha))$ cannot be compact.
\end{proof}

\begin{remark}\label{R.comp-1}\rm
Even though $\sC\in \cL(\Lambda_0(\alpha))$ is never compact, there do  exist $\alpha$ satisfying $\alpha_n\uparrow\infty$, even with $\Lambda_0(\alpha)$ nuclear, such that $\sC\colon c_0(w_k)\to c_0(w_k)$  is compact for \textit{every} $k\in\N$. Indeed, fix any
$0<\beta<1$ and set $\alpha_n:=\sum_{k=1}^n\frac{1}{k^\beta}$ for  $n\in\N$. Then  $1<\alpha_n<\alpha_{n+1}$ with $\alpha_n\uparrow\infty$. The Stolz-Ces\`aro criterion implies  $\lim_{n\to\infty}\frac{\log n}{\alpha_n}=0$, i.e., $\Lambda_0(\alpha)$ is nuclear.
Fix  $0<r<1$ and set $w(n): =r^{\alpha_n}$ for $n\in\N$. The Stolz-Ces\`aro criterion implies
\[
\lim_{n\to\infty}\frac{w(n)}{n}\sum_{m=1}^n \frac{1}{w(m)}=\lim_{n\to\infty}\frac{r^{\alpha_n}}{n}\sum_{m=1}^nr^{-\alpha_m}=0.
\]
By \cite[Corollary 2.3(ii)]{ABR-9}, $\sC\in \cL(c_0(w))$ is compact. In particular, $\sC\in \cL(c_0(w_k))$ is compact for every $k\in\N$.
\end{remark}

For a  lcHs $X$ and  $T\in \cL(X)$, the \textit{resolvent set} $\rho(T)$ of $T$ consists of all $\lambda\in\C$ such that $R(\lambda,T):=(\lambda I- T)^{-1}$ exists in $\cL(X)$. The set  $\sigma(T):=\C\setminus \rho(T)$ is called the \textit{spectrum} of $T$.
The \textit{point spectrum} $\sigma_{pt}(T)$ of $T$ consists of all $\lambda\in\C$ such that $(\lambda I-T)$ is not injective. If we need to stress the space $X$, then we  write $\sigma(T;X)$, $\sigma_{pt}(T;X)$ and $\rho(T;X)$. Given $\lambda, \mu\in \rho(T)$ the \textit{resolvent identity} $R(\lambda,T)-R(\mu,T)=(\mu-\lambda) R(\lambda,T)R(\mu,T)$ holds.
Unlike for Banach spaces, it may happen that $\rho(T)=\emptyset$ or that $\rho(T)$ is not open. This is why some authors prefer the subset $\rho^*(T)$ of $\rho(T)$ consisting of all $\lambda\in\C$ for which  there exists $\delta>0$ such that $B(\lambda,\delta):=\{z\in\C\colon |z -\lambda| < \delta\} \subseteq \rho(T)$ and $\{R(\mu,T) \colon \ \mu\in  B(\lambda,\delta)  \}$ is equicontinuous in $\cL(X)$. The advantage of $\rho^*(T)$, whenever it is non-empty, is that it is open and the resolvent map $R\colon \lambda\mapsto R(\lambda,T)$ is holomorphic from $\rho^*(T)$ into $\cL_b(X)$, \cite[Proposition 3.4]{ABR-5}. Define $\sigma^*(T):=\C\setminus \rho^*(T)$, which  is a closed set  containing $\sigma(T)$.  If $T\in \cL(X)$ with $X$  a Banach space, then $\sigma(T)=\sigma^*(T)$. In \cite[Remark 3.5(vi), p.265]{ABR-5}  a continuous linear operator $T$ on a Fr\'echet space $X$ is presented such that $\ov{\sigma(T)}\subset\sigma^*(T)$ properly. We now turn our attention to the spectrum of $\sC\in \cL(\Lambda_0(\alpha))$.

The Ces\`aro matrix $\sC$, when acting in $\C^{\N}$,  is similar to the diagonal matrix ${\rm diag}((\frac{1}{n})_n)$. Indeed, $\sC=\Delta {\rm diag}((\frac{1}{n})_n) \Delta$, with $\Delta=\Delta^{-1}=(\Delta_{n k})_{n, k \in \N} \in \cL (\C^\N)$  the lower triangular matrix  where, for each $n \in \N$, $\Delta_{n k}=(-1)^{k-1} \binom{n-1}{k-1}$, for $\ 1 \leq k \leq n$ and $\Delta_{n k}=0$ if $k>n$, \cite[pp.\ 247-249]{H49}. The dual operator $\sC'$ acts on the (row) vector space $\varphi$ of all finitely supported sequences via $x' \rightarrow x'\Delta' {\rm diag}((\frac{1}{n})_n) \Delta'$. Thus both operators have point spectrum $\Sigma$ and each eigenvalue $\frac{1}{n}$ has multiplicity $1$ with eigenvector $\Delta e_n$ (resp.\ $e_n' \Delta'$), for $n \in \N$. Moreover, $\lambda I - \sC$ (resp.\ $\lambda I - \sC')$ is invertible  for each $\lambda \in \C \setminus \Sigma$. If $X$ is a lcHs continuously contained in $\C^{\N}$ and $\sC(X)  \subseteq X$, then $\sigma_{pt}(\sC;X) = \{ \frac{1}{n} : n \in \N, \ \Delta e_n \in X \}$. In case $\varphi$ is densely contained in $X$, then $\varphi \subseteq X'$ and $\Sigma \subseteq \sigma_{pt}(\sC';X') \subseteq \sigma(C;X)$. These comments imply the following result; observe that always $\Delta e_1= \textbf{1}:=(1)_{n \in \N} \in \Lambda_0(\alpha)$.

\begin{lemma}\label{L-1} Let $\alpha$ be any sequence with $\alpha_n\uparrow\infty$. Then $1\in \sigma_{pt}(\sC; \Lambda_0(\alpha)) \subseteq \Sigma$ and $\Sigma \subseteq \sigma(\sC; \Lambda_0(\alpha))$.
\end{lemma}

\begin{prop}\label{P-K1} For $\alpha$ with $\alpha_n\uparrow\infty$ the following assertions are equivalent.
\begin{itemize}
\item[\rm (i)]  $\Lambda_0(\alpha)$ is nuclear.
\item[\rm (ii)] $\lim_{n\to\infty}\frac{\log n}{\alpha_n}=0$.
\item[\rm (iii)] $\sigma_{pt}(\sC; \Lambda_0(\alpha))=\Sigma$.
\item[\rm (iv)] $\sigma_{pt}(\sC; \Lambda_0(\alpha))\setminus\{1\}\not=\emptyset$.
\end{itemize}
\end{prop}

\begin{proof} (i)$\Leftrightarrow$(ii). See Remark \ref{R.W1}(iii).

(ii)$\Rightarrow$(iii). Since $\lim_{n\to\infty}\frac{\log n}{\alpha_n}=0$,  Lemma \ref{nuclcond} implies that $w_k\in s$ for all $k\in\N$. Hence,  $\{(n^{m-1})_{n\in\N}\colon m\in\N\}\su \Lambda_0(\alpha)$. Therefore  each vector $x^{(m)}: =\Delta e_m$, $m\in\N$ (i.e.\ $x^{(m)}_n = 0$ for $1 \leq n < m$ and
$x^{(m)}_n = \frac{(-1)^m (n-1)!}{(m-1)!(n-m)!}$ for $n \geq m$), being an eigenvector of $\sC\in \cL(\C^\N)$ corresponding to the eigenvalue $\frac{1}{m}$, belongs to  $\Lambda_0(\alpha)$; see the comments prior Lemma \ref{L-1} and \cite[Proposition 2.6]{ABR-9}. This shows $\Sigma\su \sigma_{pt}(\sC; \Lambda_0(\alpha))$. Equality now follows from  Lemma \ref{L-1}.

(iii)$\Rightarrow$(iv). Obvious.

(iv)$\Rightarrow$(ii). Let $\lambda\not=1$ belong to $\sigma_{pt}(\sC; \Lambda_0(\alpha)) \subseteq \Sigma$;  see Lemma \ref{L-1}. Then $\lambda=\frac{1}{m}$ for some $m\in\N\setminus\{1\}$. Moreover,
$(n^{m-1})_{n}\in \Lambda_0(\alpha)$; see the proof of (ii) $ \Rightarrow $ (iii) above from which it is clear that $(n^{m-1})_{n \in \N}$ behaves  asymptotically like  $\Delta e_m$. Since $(m-1)>0$, given $\ve>0$ select $k\in\N$ satisfying $\frac{2}{k(m-1)}<\ve$. Recall  $w_k(n)=e^{-\alpha_n/k}$ for all $n\in\N$. So, there is $M_k>1$ such that  $n^{m-1}e^{-\alpha_n/k}\leq M_k$ for each $n\in\N$. Hence, $(m-1)\log (n)\leq \log (M_k)+\frac{\alpha_n}{k}$ for  $n\in\N$. This implies
\[
\frac{\log n}{\alpha_n}\leq \frac{\log M_k}{m-1}\cdot\frac{1}{\alpha_n}+\frac{1}{k(m-1)}, \ \ \ n\in\N.
\]
On the other hand, since $\alpha_n\uparrow\infty$ there exists $n_0\in\N$ such that $\frac{\log M_k}{\alpha_n}<\frac{1}{k}$ for all $n\geq n_0$. It follows that $
\frac{\log n}{\alpha_n}\leq \frac{2}{k(m-1)}<\ve$, for $  n\geq n_0$, which implies (ii).
\end{proof}

\begin{prop}\label{P-K2} For $\alpha$ with $\alpha_n\uparrow\infty$ the following assertions are equivalent.
\begin{itemize}
\item[\rm (i)] $0\not\in \sigma(\sC; \Lambda_0(\alpha))$.
\item[\rm (ii)] For each $k$ there exists $l>k$ such that $\sup_n (\log (n) - (\frac{1}{k}-\frac{1}{l}) \alpha_n) < \infty$.
\item[\rm (iii)] $\Lambda_0(\alpha)$ is nuclear.
\end{itemize}
\end{prop}

\begin{proof} (i)$\Rightarrow$(ii). Suppose that $0\not\in\sigma(\sC;\Lambda_0(\alpha))$, i.e., the inverse operator $\sC^{-1}\colon \C^\N\to \C^\N$, given by $y=(y_n)_{n\in\N}\mapsto \sC^{-1}(y)=(ny_n-(n-1)y_{n-1})_{n\in\N}$, with $y_0:=0$, is  continuous on $\Lambda_0(\alpha)$. In view of \eqref{eq.seminorme} this holds  if and only if
\begin{equation}\label{eq.con-inversa}
\forall k \ \exists l>k \ \exists D>0: \sup_n w_k(n) |n y_n - (n-1) y_{n-1} | \leq D \sup_n w_l(n) |y_n|,
\end{equation}
for all  $y \in \Lambda_0(\alpha)$. Given $n \in \N$, set $y=e_n$ to be the $n$-th canonical basis vector of $\Lambda_0(\alpha)$. Then  \eqref{eq.con-inversa} yields
$n w_k(n) = \max\{n w_k(n),nw_k(n+1)\} \leq D w_l(n)$.
Since $n\in\N$ is arbitrary, by taking logarithms  we have shown that
$$
\forall k \ \exists l>k \ \exists D>1 \ \forall n: \log(n) - \left(\frac{1}{k}-\frac{1}{l}\right) \alpha_n \leq \log D,
$$
which is precisely the condition stated in (ii).

(ii)$\Rightarrow$(i).
Recall
that $w_k(n)=\exp(- \alpha_n/k)$ for all $n,k \in \N$. Fix now
 $k\in\N$. Then there exist $l>k$ and $M>0$ with
$\log n \leq (\frac{1}{k}-\frac{1}{l}) \alpha_n + M$ for each $n \in \N$. Hence, $n w_k(n)\leq e^M w_l(n)$ for each $n \in\N$. Since $w_k$ is decreasing, this implies, for every $y\in \Lambda_0(\alpha)$ and $n\in\N$, that
\begin{eqnarray*}
& & w_k(n)|ny_n-(n-1)y_{n-1}| \leq  nw_k(n)|y_n|+(n-1)w_k(n)|y_{n-1}|\\
& & \qquad\leq  nw_k(n)|y_n|+(n-1)w_k(n-1)|y_{n-1}|\leq2 e^M \sup_n w_l(n)|y_n|.
\end{eqnarray*}
Hence,  \eqref{eq.con-inversa} is satisfied and so  $C^{-1}$ is continuous on $\Lambda_0(\alpha)$, i.e., $0\not\in\sigma(\sC;\Lambda_0(\alpha))$.

(iii)$\Rightarrow$(ii). Observe, for any given $k\in\N$, that
\[
\log (n)-\left(\frac{1}{k}-\frac{1}{k+1}\right)\alpha_n=\alpha_n\left(\frac{\log n}{\alpha_n}-\frac{1}{k(k+1)}\right),\quad n\in\N.
\]
Since $\Lambda_0(\alpha))$ is nuclear,  $\lim_{n\to\infty}\frac{\log n}{\alpha_n}=0$  and so $\left(\frac{\log n}{\alpha_n}-\frac{1}{k(k+1)}\right)<0$ for all $n$ large enough. Hence, the condition stated in (ii) follows for  $l:=(k+1)$.

(ii)$\Rightarrow$(iii). Fix $k\in\N$. Then there exist $l>k$ and $M>0$ with
\[
\log (n) \leq \left(\frac{1}{k}-\frac{1}{l}\right) \alpha_n + M<\frac{\alpha_n}{k}+M,\quad n \in \N.
\]
Hence,
$\frac{\log n}{\alpha_n}<\frac{1}{k}+\frac{M}{\alpha_n}$, for $n\in\N$.
Since $\frac{M}{\alpha_n}\to 0$, there exists $n_0\in\N$ such that $\frac{\log n}{\alpha_n}<\frac{2}{k}$, for $n\geq n_0$.
The arbitrariness of $k$ in $\N$ yields that $\lim_{n\to\infty}\frac{\log n}{\alpha_n}=0$. Hence, the condition stated in (iii) follows.
\end{proof}

The formal operator of differentiation acts on $\C^{\N}$ via
$$
D(x_1,x_2,x_3,...):=(x_2,2 x_3,3 x_4,...), \ \ \ \ x=(x_1,x_2,x_3,...) \in \C^{\N}.
$$
The inverse operator $\sC^{-1}$ of  $\sC$ then coincides with the formal differential operator $(1-z)(1 + z D(Z))$ on the algebra of all formal power series $Z=Z(z)$. The referee suggested there should be a connection between the nuclearity of $\Lambda_0(\alpha)$ and the continuity of $D$ on $\Lambda_0(\alpha)$. The following result and Remark \ref{R.diff} address this point.
Recall that $\Lambda_0 (\alpha)$ is \textit{shift stable} if $ \lim \sup_{n \to \infty } \frac{\alpha_{n+1}}{\alpha_n} < \infty $, \cite{V3}.

\begin{prop}\label{p.differentiation} For $\alpha$ with $\alpha_n\uparrow\infty$ the following assertions are equivalent.
\begin{itemize}
\item[\rm (i)] The differentiation operator $D: \Lambda_0(\alpha) \rightarrow \Lambda_0(\alpha)$ is continuous.
\item[\rm (ii)] For each $k$ there exist $l>k$ and $M>0$ such that
\begin{equation}\label{26}
    n w_k(n) \leq M w_l(n+1),   \qquad    n \in \N
\end{equation}
\item[\rm (iii)]
$\Lambda_0 (\alpha)$ is both nuclear and shift stable.
\end{itemize}
\end{prop}
\begin{proof}
(i)$\Rightarrow$(ii). If $D: \Lambda_0(\alpha) \rightarrow \Lambda_0(\alpha)$ is continuous, then for each $k$ there exist $l>k$ and $M>0$ such that
$\sup_n w_k(n)|(Dx)_n| \leq M \sup_n w_l(n)|x_n|$, for each $x \in \Lambda_0(\alpha)$. Select $x: = e_j$, for $j \geq 2$. Since
$D e_j=(j-1) e_{j-1}$, it follows that $(j-1) w_k(j-1) \leq M w_l(j)$ for each $j \geq 2,$ which is precisely (ii).

(ii)$\Rightarrow$(i). Given $k$, select $l$ and $M>0$ via (ii). For $x \in \Lambda_0(\alpha)$,  $n \in \N$ we have
$$
w_k(n)|(Dx)_n| = w_k(n)n|x_{n+1}| \leq M w_l(n+1)|x_{n+1}| \leq M p_l(x), \ \ \ n \in \N,
$$
that is, $p_k (D x) \le M p_l (x)$. This shows that $ D \in \Lambda_0 (\alpha)$.

(i)$\Rightarrow$(iii).
Fix $k \in \N$. According to \eqref{26} there exist $l > k $ and $M> 1$ such that $n e^{-\alpha_n/k} \le M e^{-\alpha_{n+1}/l}
\le M e^{-\alpha_n / l}$, for $n \in \N $. Taking logarithms implies the inequality
$
\log (n) - \left(\frac 1 k - \frac 1 l \right) \!\alpha_n \le \log (M) $, for $ \in \N $.
According to Proposition \ref{P-K2} the space $ \Lambda_0(\alpha)$ must be  nuclear.

For $k = 1 $, choose $l > k $ and $\widetilde{M} > 1 $ to satisfy \eqref{26}, i.e., $n e^{- \alpha_n} \le \widetilde{M} e^{-\alpha_{n+1}/l}$
for $n \in \N $. With $\beta := \log (\widetilde{M})$ it follows that $\log (n) - \alpha_n \le \beta - \frac{\alpha_{n+1}}{l}$ which yields
the inequality
$
\frac{\alpha_{n+1}}{\alpha_n} \le l + \frac{\beta l}{\alpha_n} - \frac{\log (n)}{\alpha_n}  l $, for $n \in \N $.
Since $\lim_{n \to \infty} \frac{\log (n)}{\alpha_n} =0$, we see that $\lim \sup_{n \to \infty } \frac{\alpha_{n+1}}{\alpha_n}
< \infty $, that is, $ \Lambda_0(\alpha)$ is shift stable.

(iii)$\Rightarrow$(ii).
Fix $k \in \N $. Choose $l > k $ and $R > 1 $ with
$
\log (n) -  \left(\frac{1}{k}-\frac{1}{l}\right)  \alpha_n \le R $, for $ n \in \N$;
see Proposition \ref{P-K2}. With $ M : = e^R$ it follows that $n w_k (n) \le  M w_l(n)$ for $n \in \N$.
Since  $\Lambda_0(\alpha)$ is shift stable, there is $ s \in \N $ such that  $\alpha_{n+1} \leq s \alpha_n$ for $n \in \N$,
which implies that $w_l (n) \le w_{sl} (n+1)$ for $n \in \N$. Accordingly, with $L:= sl $ we have $n w_k(n) \leq M w_L (n+1)$
for $n \in \N$. So, (ii) satisfied.
\end{proof}

\begin{remark}\label{R.diff}\rm
There exist nuclear spaces $\Lambda_0(\alpha)$ such that $D$ is \textit{not} continuous on $\Lambda_0(\alpha)$. Indeed,
let $\alpha_n := n^n $ for $ n \in \N $. Then $\Lambda_0(\alpha)$ is nuclear but,   not shift stable.
Proposition \ref{p.differentiation} implies that $D \not\in \cL (\Lambda_0(\alpha))$. On the other hand, for $\alpha_n :=
\log (n)$, $n \in \N $, the space $\Lambda_0 (\alpha)$ is shift stable but, not nuclear; again $D \not\in \cL (\Lambda_0(\alpha))$.
\end{remark}

\begin{prop}\label{p.sprectrum} Let $\alpha$ be any sequence with $\alpha_n\uparrow \infty$.
\begin{itemize}
\item[\rm  (i)] %If $\lim_{n\to\infty}\frac{\log n}{\alpha_n}=0$,
Suppose that $\Lambda_0(\alpha)$ is nuclear. Then
\begin{equation*}
\sigma(\sC; \Lambda_0(\alpha))=\sigma_{pt}(\sC; \Lambda_0(\alpha))=\Sigma.
\end{equation*}
\item[\rm (ii)] If, in addition to $\Lambda_0(\alpha)$ being nuclear, also $v(\alpha)>0$, then
\begin{equation*}
\sigma^*(\sC; \Lambda_0(\alpha))=\ov{\sigma(\sC; \Lambda_0(\alpha))}=\Sigma_0.
\end{equation*}
\end{itemize}
\end{prop}

\begin{proof} (i) By Proposition \ref{P-K1} we have  $\Sigma= \sigma_{pt}(\sC; \Lambda_0(\alpha))$. Thus, $\Sigma\su \sigma(\sC; \Lambda_0(\alpha))$.  Moreover, $0 \notin \sigma(\sC; \Lambda_0(\alpha))$ by Proposition \ref{P-K2}. It remains to verify that $\lambda \notin
\sigma(\sC; \Lambda_0(\alpha))$ for each $\lambda \notin \Sigma_0$. The proof follows the lines of that of Theorem 3.4, Step 4, in \cite{ABR-9}.
Fix $\lambda\in \C\setminus\Sigma_0$, in which case $\lambda\in \rho(\sC; \C^N)$. We recall the formula for the inverse operator $(\sC-\lambda I)^{-1}\colon \C^\N\to \C^\N$, \cite[p.266]{R}. For $n\in\N$ the $n$-th row of the matrix for $(\sC-\lambda I)^{-1}$ has the entries
\[
\frac{-1}{n\lambda^2\prod_{k=m}^n\left(1-\frac{1}{\lambda k}\right)},\quad 1\leq m<n,
\]
\[
\frac{n}{1-n\lambda}=\frac{1}{\frac{1}{n}-\lambda}, \quad m=n,
\]
and all the other entries in row $n$ are equal to $0$. So,
we can write
\begin{equation}\label{e:dec}
(\sC-\lambda I)^{-1}=D_\lambda-\frac{1}{\lambda^2}E_\lambda,
\end{equation}
where the diagonal operator $D_\lambda=(d_{nm})_{n,m\in\N}$ is given by $d_{nn}:=\frac{1}{\frac{1}{n}-\lambda}$ and $d_{nm}:=0$ if $n\not=m$. The operator $E_\lambda=(e_{nm})_{n,m\in\N}$ is then the lower triangular matrix with $e_{1m}=0$ for all $m\in\N$, and for every $n\geq 2$ with $e_{nm}:=\frac{1}{n\prod_{k=m}^n\left(1-\frac{1}{\lambda k}\right)}$ if $1\leq m<n$ and $e_{nm}:=0$ if $m\geq n$.

Since $\lambda\not\in \Sigma_0$, it follows from \eqref{eq.seminorme} that $D_\lambda\in \cL(\Lambda_0(\alpha))$.
By \eqref{e:dec} it  remains to show that $E_\lambda\in \cL(\C^\N)$ maps $\Lambda_0(\alpha)$ continuously into $\Lambda_0(\alpha)$.
To this end we observe, for every $k\in\N$, that $c_0(w_k)$ is
isometrically isomorphic to $c_0$ via the linear multiplication operator $\Phi_k\colon c_0(w_k)\to c_0$ given by $\Phi_k(x):=(w_k(n)x_n)_n$, for $x=(x_n)_n\in c_0(w_k)$. If we can show that  ${E}_{\lambda}$ maps $c_0(w_{k+1})$ into $c_0(w_k)$ continuously, for all $k\in\N$, then $E_\lambda$ will map $\Lambda_0(\alpha)$ into itself continuously. So, it suffices to show that    $\tilde{E}_{\lambda,k}:=\Phi_kE_\lambda \Phi_{k+1}^{-1}\in \cL(c_0)$ for all $k\in\N$.

Fix $k\in\N$. Now, $\tilde{E}_{\lambda,k}$
 is  the restriction to $c_0$ of  the operator on $\C^\N$ given by
\[
(\tilde{E}_{\lambda,k}(x))_n=w_k(n)\sum_{m=1}^{n-1}\frac{e_{nm}}{w_{k+1}(m)}x_m, \quad x\in \C^\N,\ n\in\N,
\]
with $(\tilde{E}_{\lambda,k}(x))_1:=0$.
Observe that $\tilde{E}_{\lambda,k}=(\tilde{e}^{(k)}_{nm})_{n,m\in\N}$ is the lower triangular matrix given by $\tilde{e}^{(k)}_{1m}=0$ for $m\in\N$ and $\tilde{e}^{(k)}_{nm}=\frac{w_k(n)}{w_{k+1}(m)}e_{nm}$ for $n\geq 2$ and $m\in\N$.
 So, we need to verify  that $\tilde{E}_{\lambda,k}\in \cL(c_0)$. To prove this, set $\alpha:={\rm Re}\left(\frac{1}{\lambda }\right)$.
Since $\lambda\in \C\setminus\Sigma_0$, by both Lemma 3.3 and the proof of Step 4 in the proof of Theorem 3.4 in \cite{ABR-9} there exist $c>0$ and $C>0$ such that
\begin{eqnarray}\label{e.Boun}
& &\label{e.Boun1} \frac{c}{n^{1-\alpha}}\leq |e_{n1}|\leq \frac{C}{n^{1-\alpha}},\quad n\geq 2,\\
& &\label{e.Boun2}\frac{c}{n^{1-\alpha}m^\alpha}\leq  |e_{nm}|\leq \frac{C}{n^{1-\alpha}m^\alpha},\quad 2\leq m<n.
\end{eqnarray}
So,  by \cite[Theorem 4.51-C]{T} to prove that $\tilde{E}_{\lambda,k}\in \cL(c_0)$ we need to verify, for each $k\in\N$, that the following two conditions are satisfied:
\begin{itemize}
\item[\rm (a)] $\lim_{n\to\infty}\tilde{e}^{(k)}_{nm}=0$, for each $m\in\N$ and
\item[\rm (b)] $\sup_{n\in\N}\sum_{m=1}^\infty |\tilde{e}^{(k)}_{nm}|<\infty$.
\end{itemize}
First observe that \eqref{e.Boun1} and \eqref{e.Boun2} imply for every $m,\ n\in\N$ that
\[
|\tilde{e}^{(k)}_{nm}|=\frac{w_k(n)}{w_{k+1}(m)}|e_{nm}|\leq C'_m \frac{r_k^{\alpha_n}}{n^{1-\alpha}}.
\]
But,  for each fixed $m\in\N$, we have via Lemma \ref{nuclcond} that $\frac{r_k^{\alpha_n}}{n^{1-\alpha}}\to 0$ for $n\to\infty$ and hence, condition (a) is satisfied.

Next, fix $k\in\N$.  Then \eqref{e.Boun1} and \eqref{e.Boun2} imply, for every $n\in\N$, that
\begin{equation}\label{eq.condb}
\sum_{m=1}^\infty |\tilde{e}^{(k)}_{nm}|=\sum_{m=1}^{n-1}\frac{w_k(n)}{w_{k+1}(m)}|e_{nm}|\leq C\frac{1}{n}\sum_{m=1}^{n-1}\frac{r_k^{\alpha_n}}{r_{k+1}^{\alpha_m}}\frac{n^\alpha}{m^\alpha}.
\end{equation}
Suppose first that $\alpha<1$. Since $0<r_{k+1}<1$ and $(\alpha_n)_n$ is an increasing sequence, we have $r_{k+1}^{\alpha_m}\geq r_{k+1}^{\alpha_n}$ for $n\in\N$ and $m=1,\ldots, n-1$, and so $\frac{r_k^{\alpha_n}}{r_{k+1}^{\alpha_m}}\leq \left(\frac{r_k}{r_{k+1}}\right)^{\alpha_n}<1$ for $n\in\N$ and $m=1,\ldots, n-1$. It follows for every $n\in\N$ that
\begin{equation}\label{eq.condb1}
\frac{1}{n}\sum_{m=1}^{n-1}\frac{r_k^{\alpha_n}}{r_{k+1}^{\alpha_m}}\frac{n^\alpha}{m^\alpha}\leq \frac{1}{n}\sum_{m=1}^{n-1}\frac{n^\alpha}{m^\alpha}\leq \frac{1}{n^{1-\alpha}}\sum_{m=1}^{n-1}\frac{1}{m^\alpha}\leq \max\left\{1,\frac{1}{1-\alpha}\right\}<\infty
\end{equation}
whenever $\alpha<1$; see the proof of Corollary 3.6 in \cite{ABR-9}. So,  \eqref{eq.condb} and \eqref{eq.condb1} ensure that  condition (b) is surely satisfied if $\alpha<1$.

Consider now $\alpha\geq 1$. Then,  for every $n\in\N$, we have (as $\frac{1}{r_{k+1}}>0$) that
\begin{eqnarray}\label{eq.condb2}
& & \frac{1}{n}\sum_{m=1}^{n-1}\frac{r_k^{\alpha_n}}{r_{k+1}^{\alpha_m}}\frac{n^\alpha}{m^\alpha}
=n^{\alpha-1}r_k^{\alpha_n}\sum_{m=1}^{n-1}\left(\frac{1}{r_{k+1}}\right)^{\alpha_m}\frac{1}{m^\alpha}\nonumber\\
& & \leq
n^{\alpha-1}r_k^{\alpha_n}(n-1)\left(\frac{1}{r_{k+1}}\right)^{\alpha_n}\leq n^\alpha\left(\frac{r_k}{r_{k+1}}\right)^{\alpha_n}.
\end{eqnarray}
But, by Lemma \ref{nuclcond}, $n^\alpha\left(\frac{r_k}{r_{k+1}}\right)^{\alpha_n}\to 0$ for $n\to\infty$ (because $0<\frac{r_k}{r_{k+1}}<1$ and $\frac{\log n}{\alpha_n}\to 0$ for $n\to\infty$ via Remark \ref{R.W1}(iii)). So,  \eqref{eq.condb} and \eqref{eq.condb2} ensure that  condition (b) is also satisfied if $\alpha\geq 1$.
%%%%%%%%%%%%%%

(ii) According to part (i), $\sigma(\sC; \Lambda_0(\alpha))=\Sigma$. For $k\in\N$ fixed, $\frac{w_k(n+1)}{w_k(n)}=r_k^{\alpha_{n+1}-\alpha_n}\leq r_k^{v(\alpha)}$ for all $n\in\N$. Hence,
\[
\limsup_{n}\frac{w_k(n+1)}{w_k(n)}\leq r_k^{v(\alpha)}<1.
\]
It follows from \cite[Proposition 2.7]{ABR-9} that $\sigma(\sC; c_0(w_k))=\Sigma_0$. Accordingly,
\[
\cup_{k\in\N}\sigma(\sC; c_0(w_k))=\Sigma_0=\ov{\Sigma}=\ov{\sigma(\sC; \Lambda_0(\alpha))}.
\]
Now \cite[Lemma 2.1]{ABR-1} implies that $\sigma^*(\sC; \Lambda_0(\alpha))=\ov{\sigma(\sC; \Lambda_0(\alpha))}=\Sigma_0$.
\end{proof}

\begin{corollary}\label{C.Su-1} For  $\alpha$ with $\alpha_n\uparrow\infty$ the following assertions are equivalent.
\begin{itemize}
\item[\rm (i)]  $\Lambda_0(\alpha)$ is nuclear.
\item[\rm (ii)]  $\sigma(\sC; \Lambda_0(\alpha))=\sigma_{pt}(\sC; \Lambda_0(\alpha))$.
\item[\rm (iii)]  $\sigma(\sC; \Lambda_0(\alpha))=\Sigma$.
\end{itemize}
\end{corollary}

\begin{proof} (i)$\Rightarrow$(ii). This is part of Proposition \ref{p.sprectrum}(i).

(ii)$\Rightarrow$(i). The equality in (ii) together with Lemma \ref{L-1} imply that $\Sigma\su \sigma_{pt}(\sC; \Lambda_0(\alpha))$ and hence, by Proposition \ref{P-K1}, the space $\Lambda_0(\alpha)$ is nuclear.

(i)$\Rightarrow$(iii). Clear from Proposition \ref{p.sprectrum}(i).

(iii)$\Rightarrow$(i). The equality in (iii) implies that $0\not\in\sigma(\sC;\Lambda_0(\alpha))$ and so $\Lambda_0(\alpha)$ is nuclear; see Proposition \ref{P-K1}.
\end{proof}

The identity $\sC= \Delta \mbox{diag} ((\frac 1 n)_n)  \Delta $ holds in $\cL (\C^\N)$ with all three operators $\sC, \Delta$ and $\mbox{diag} ((\frac 1 n)_n)$ continuous;
see the discussion prior to Lemma \ref{L-1}. For every sequence $\alpha_n \uparrow \infty $, both of the operators $\sC$ and $\mbox{diag} ((\frac 1 n )_n) $ also
belong to $\cL (\Lambda_0 (\alpha))$. Since the columns $\{\Delta e_n : n \in \N \}$ of $\Delta $ are the distinct eigenvectors of $\sC \in \cL (\C^N)$,
it follows from the discussion prior to Lemma \ref{L-1} and Corollary \ref{C.Su-1} that a necessary condition for $\Delta$ to act in $\Lambda_0 (\alpha)$
is the nuclearity of $\Lambda_0 (\alpha)$. However, this condition alone is not sufficient for the continuity of $\Delta$.

\begin{prop} \label{pr14}
Let $\Lambda_0(\alpha)$ be nuclear. The following assertions are equivalent.
\begin{enumerate}
  \item[\rm (i)]
  $\Delta : \Lambda_0(\alpha) \to \Lambda_0(\alpha)$ is continuous.
  \item[\rm (ii)]
 For each $k \in \N $ there exists $l \ge k$ such that
 \begin{equation}\label{new}
  \sup_n \sum^n_{m=1}  \frac{w_k (n)}{w_l (m)}
\binom{n-1}{m-1}  < \infty .
  \end{equation}

 \item[\rm (iii)]
 $\lim_{n \to \infty } \frac{n}{\alpha_n} =0$.
\end{enumerate}
\end{prop}

\begin{proof} (i)$\Leftrightarrow$(ii).
By definition $\Delta  \in \cL (\Lambda_0(\alpha))$ if and only if for each $k \in \N $ there exists $l \ge k$
such that $\Delta : c_0 (w_l) \to c_0 (w_k)$ is continuous. Since the linear map $\Phi_m : x = (x_n)_{n \in \N}
\mapsto (w_m (n)x_n)_{n \in \N }$ is an isometric isomorphism from the weighted Banach space $c_0 (w_m)$
onto $c_0$, for each $m \in \N$, it follows that $\Delta : c_0 (w_l) \to c_0 (w_k)$ is continuous if and only if
$T_{k,l} : c_0 \to c_0$ is continuous, where $T_{k,l} := \Phi_k \Delta \Phi_l^{-1}$ is given by the lower triangular
matrix $T_{k,l} = (t_{n,m})_{n,m \in \N }$ with $|t_{n,m}| = \frac{w_k (n)}{w_l (m)} \binom{n-1}{m-1} $
for $ 1 \le m \le n $ and $t_{n,m} =0$ otherwise. By Theorem 4.51-C of \cite{T},
$T_{k,l} \in \cL (c_0)$ if and only if both $\lim_{n \to \infty } |t_{n,m}| =0$ for each $m \in \N $ (which is equivalent
to the nuclearity of $\Lambda_0 (\alpha)$) and that \eqref{new} is satisfied.

(ii)$\Rightarrow$(iii).
Fix $k \in \N $ and select $l \ge k$ to satisfy \eqref{new}. Then there is $R_k > 1 $ such that
$$
w_k (n) \sum^n_{m=1}  \binom{n-1}{m-1}
\frac{1}{w_l (m)} \le R_k, \qquad n \in \N .
$$
Since $\frac{1}{w_l (m)} > 1 $ for all $m, l \in \N $ and $\sum^n_{m=1}  \binom{n-1}{m-1} = 2^{n-1}$, it follows that
$$
2^{n-1} e^{- \alpha_n/k} = 2 ^{n-1} w_k (n) \le w_k (n) \sum^n_{m=1}
\binom{n-1}{m-1} \frac{1}{w_l(m)} \le R_k, \qquad n \in \N .
$$
Via the identity $2^{n-1} e^{-\alpha_n/k} = \exp \left((n-1) \log (2) - \frac{\alpha_n}{k}\right)$ this inequality
can be solved to yield
$$
\frac{n}{\alpha_n }  \le \left(  1+ \frac{\log (R_k)}{\log (2)}  \right) \frac{1}{\alpha_n} + \frac{1}{k \log (2)} ,
\qquad n \in \N ,
$$
which implies that $\lim_{n \to \infty }  \frac{n}{\alpha_n}$.

(iii)$\Rightarrow$(ii).
Fix $k \in \N $ and set $l:= 2 k$. Since $w_l$ is decreasing and $\frac{w_k (n)}{w_{2k}(n)} = e^{- \alpha_n/2k}$
we have, for each $n \in \N $, that
\begin{eqnarray*}
&& \sum^n_{m=1} \frac{w_k (n)}{w_l (m)}
\binom{n-1}{m-1} \le \frac{w_k (n)}{w_{2k} (n)} \sum^n_{m=1}
\binom{n-1}{m-1}  \le e^{- \alpha_n/2k} 2^n \\
&& = \exp \left(\alpha_n \left (\frac{n \log (2)}{\alpha_n} - \frac{1}{2k}\right)\right) .
\end{eqnarray*}
It is then  clear from $\alpha_n \uparrow \infty $ and $\lim_{n \to \infty } \frac{n \log (2)}{\alpha_n} =0$ that the
left-side of the previous inequality converges to $0$ for $n \to \infty $. In particular, \eqref{new} is satisfied.
\end{proof}

\begin{remark} \rm
(i)
For each $\beta > 0 $ consider $\alpha_\beta (n) := n^\beta$, for $n \in \N$. Then $\lim_{n \to \infty } \frac{ \log (n)}{\alpha_\beta (n)} =0$
and so $\Lambda_0 (\alpha_\beta)$ is nuclear for every $\beta > 0 $. But, $\lim_{n \to \infty } \frac{n }{\alpha_\beta (n)} =0$ if
and only if $\beta > 1 $, i.e., $\Delta \in \cL (\Lambda_0 (\alpha_\beta))$ if and only if $\beta > 1$.

(ii)
The continuity of the operators $\Delta$ and $D$  in $\Lambda_0 (\alpha)$ is unrelated. Indeed, $D$ is continuous on
$\Lambda_0 (\alpha_\beta)$ for every $\beta \in (0,1)$ whereas $\Delta $ is not. On the other hand, by Proposition \ref{pr14}, $\Delta $ is continuous on
$\Lambda_0 (\alpha)$ for $\alpha_n := n^n$, $n \in \N$, but $D$ fails to be continuous on this space; see Remark \ref{R.diff}.

\end{remark}

Recall that $w_k(n)=e^{-\alpha_n/k}$ for $k,\ n\in\N$.
In order to formulate the following results, given a sequence $\alpha$ with $\alpha_n\uparrow\infty$ define
\begin{equation}\label{eq.condE}
S_k(\alpha):=\left\{s\in\R\colon \sum_{n=1}^\infty\frac{1}{n^sw_k(n)}=\sum_{n=1}^\infty\frac{e^{\alpha_n/k}}{n^s}<\infty\right\},\quad k\in\N.
\end{equation}
It follows from \eqref{eq.condE} and $w_k\leq w_{k+1}$ that $S_k(\alpha)\su S_{k+1}(\alpha)$ for all $k\in\N$.
If $S_{k_0}(\alpha)\not=\emptyset$ for some $k_0\in\N$, then $S_k(\alpha)\not=\emptyset$ for all $k\geq k_0$ and we define $s_0(k):=\inf S_k(\alpha)$, in which case $s_0(k)\geq s_0(k+1)$ for all $k\geq k_0$. Moreover, $s_0(k)\geq 1$ for all $k\geq k_0$; see the inequality (3.1) in \cite{ABR-9}. Observe that $\alpha_n= \beta \log n,$ for $ \beta>0,$ satisfies
$S_1(\alpha)=(1+\beta,\infty) \neq \emptyset$.

\begin{lemma}\label{L.S} For any sequence  $\alpha$ with $\alpha_n\uparrow\infty$ the following assertions hold.
\begin{itemize}
\item[\rm (i)] $S_k(\alpha)\not=\emptyset$ for some  $k\in\N$ if and only if $S_k(\alpha)\not=\emptyset$ for every $k\in\N$.
\item[\rm (ii)] If  $S_1(\alpha)\not=\emptyset$, then $\Lambda_0(\alpha)$ is not nuclear and $s_0(\alpha):=\inf_{k\in\N}s_0(k)=1$.
\end{itemize}
\end{lemma}

\begin{proof} (i) Suppose that $S_k(\alpha)\not=\emptyset$ for some  $k\in\N$. From the discussion prior to the lemma it is clear that $S_r(\alpha)\not=\emptyset$ for all $r\geq k$. If $k=1$, then $S_r(\alpha)\not=\emptyset$ for every $r\in\N$. So, assume that $k\geq 2$. According to \eqref{eq.condE} there exists $t\geq 1$ satisfying $\sum_{n=1}^\infty\frac{e^{\alpha_n/k}}{n^t}<\infty$ and hence, for some $n_0\in\N$, we have $\frac{e^{\alpha_n/k}}{n^t}\leq 1,$ for $n\geq n_0$.
Since $x^\beta\leq x$ for every $x\in [0,1]$ and $\beta>1$, it follows that
\[
\frac{e^{\alpha_n/(k-1)}}{n^{tk/(k-1)}}= \left(\frac{e^{\alpha_n/k}}{n^t}\right)^{k/(k-1)}\leq \frac{e^{\alpha_n/k}}{n^t},\quad n\geq n_0.
\]
So, $\sum_{n=1}^\infty \frac{e^{\alpha_n/(k-1)}}{n^{tk/(k-1)}}<\infty$ which shows that $\frac{tk}{(k-1)}\in S_{k-1}(\alpha)$, i.e., $S_{k-1}(\alpha)\not=\emptyset$. This argument can be repeated $ (k-1)$   times to conclude $S_r(\alpha)\not=\emptyset$ for all $1\leq r<k$. Hence, $S_r(\alpha)\not=\emptyset$ for every $r\in\N$. The converse is obvious.

(ii)  If $\Lambda_0(\alpha)$ is nuclear, then  Lemma \ref{nuclcond}  implies  that $w_1\in s$ and hence, $S_1(\alpha)=\emptyset$, \cite[Proposition 3.1(iii)]{ABR-9}. This contradicts the hypothesis that $S_1(\alpha)\not=\emptyset$ and so $\Lambda_0(\alpha)$ is not nuclear.
Since $S_1(\alpha)\not=\emptyset$, there exists $t\geq 1$ satisfying $\sum_{n=1}^\infty\frac{e^{\alpha_n}}{n^t}<\infty$. In particular, there exists $n_0\in\N$ such that $e^{\alpha_n}\leq n^t$ for all $n\geq n_0$. Fix any $\ve>0$ and select $k_0\in\N$ satisfying $\frac{t}{k_0}<\frac{\ve}{2}$. Then it follows, for every $k\geq k_0$ and $n\geq n_0$, that $e^{\frac{\alpha_n}{k}}<n^{\frac{t}{k}}$. This yields, for every $k\geq k_0$, that
\[
\sum_{n=n_0}^\infty\frac{1}{n^{\frac{t}{k}+1+\frac{\ve}{2}}w_k(n)}=\sum_{n=n_0}^\infty\frac{e^\frac{\alpha_n}{k}}{n^{\frac{t}{k}+1+\frac{\ve}{2}}}\leq \sum_{n=n_0}^\infty\frac{1}{n^{1+\frac{\ve}{2}}}<\infty.
\]
Accordingly, $\frac{t}{k}+1+\frac{\ve}{2}\in S_k(\alpha)$ for all $k\geq k_0$ and hence, $1\leq s_0(k)\leq \frac{t}{k}+1+\frac{\ve}{2}$ for all $k\geq k_0$. But, $\frac{t}{k}+1+\frac{\ve}{2}<1+\ve$ and so  $1\leq s_0(k)<1+\ve$ for all $k\geq k_0$. Since $\ve>0$ is arbitrary, this implies that $s_0(\alpha)=\lim_{k\to\infty}s_0(k)=1$.
\end{proof}

\begin{remark}\label{example}\rm There exists $\alpha$ with $\alpha_n\uparrow\infty$ for which $S_1(\alpha)=\emptyset$ but, $\Lambda_0(\alpha)$ is not nuclear. Indeed,
let $(j(k))_{k\in\N} \subseteq \N$ be the sequence given by $j(1):=1\quad {\rm and}\quad j(k+1):=2(k+1)(j(k))^k,$ for $k\geq 1$. Observe that $j(k+1)>k(j(k))^k+1$ for all $k\in\N$. Define $\beta=(\beta_n)_{n\in\N}$ via $\beta_n:=k(j(k))^k$ for $n=j(k),\ldots, j(k+1)-1$. Then $\beta$ is non-decreasing with $\beta_n\to\infty$. The claim is that
\begin{equation}\label{eq.serie}
\sum_{n=1}^\infty\frac{\beta_n}{n^t}=\infty,\quad t\in\R.
\end{equation}
To see this, fix  $t\in\R$ and choose $k\in\N$ satisfying $k>t$. Then $\frac{\beta_{j(k)}}{(j(k))^t}=\frac{k(j(k))^k}{(j(k))^t}=k(j(k))^{k-t}\geq k$ for all $k\in\N$ and so the subsequence $(\frac{\beta_{j(k)}}{(j(k))^t})_{k\in\N}$ of $(\frac{\beta_n}{n^t})_{n\in\N}$ satisfies $\sup_{k\in\N}\frac{\beta_{j(k)}}{(j(k))^t}=\infty$. In particular, $\sum_{n=1}^\infty\frac{\beta_n}{n^t}=\infty$. Moreover,
\begin{equation}\label{eq.lim}
\lim_{n\to\infty}\frac{\log n}{\log \beta_n}\not=0.
\end{equation}
Indeed, this follows immediately from the fact that, for every $k>1$, we have
\[
\frac{\log (j(k+1))-1)}{\log \beta_{j(k+1)-1}}=\frac{\log (j(k+1))-1)}{\log (k(j(k))^k)}>\frac{\log (k(j(k))^k)}{\log (k(j(k))^k)}=1.
\]
Next, let $\gamma=(\gamma_n)_{n\in\N}$ be any strictly increasing sequence satisfying $2<\gamma_n\uparrow 3$. Define  $\alpha_n:=\log (\beta_n+\gamma_n)$, for $n\in\N$. Then  $1<\alpha_n\uparrow\infty$. The claim is that  $S_1(\alpha)=\emptyset$ and $\lim_{n\to\infty}\frac{\log n}{\alpha_n}\not=0$. To see this fix $t\in\R$. Then \eqref{eq.serie} implies that
\[
\sum_{n=1}^\infty\frac{e^{\alpha_n}}{n^t}=\sum_{n=1}^\infty\frac{\beta_n+\gamma_n}{n^t}\geq \sum_{n=1}^\infty\frac{\beta_n}{n^t}=\infty,
\]
that is, $t\not\in S_1(\alpha)$. Accordingly, $S_1(\alpha)=\emptyset$.
On the other hand, for all $n>1$ we have $(\beta_n+\gamma_n)<(\beta_n+3)$ and so
\[
\frac{\log n}{\log (\beta_n)}>\frac{\log n}{\alpha_n}=\frac{\log n}{\log (\beta_n+\gamma_n)}>\frac{\log n}{\log (\beta_n+3)}.
\]
But,  $\frac{\log n}{\log (\beta_n+3)}\simeq \frac{\log n}{\log\beta_n}$ for $n\to\infty$ and hence, by \eqref{eq.lim}, it follows that  $\lim_{n\to\infty}\frac{\log n}{\alpha_n}\not=0$. In particular,  $\Lambda_0(\alpha)$ cannot be nuclear; see Proposition \ref{P-K1}.
\end{remark}

For each $r\geq 1$, define the open disc $D(r):=\left\{\lambda\in\C\colon \left|\lambda-\frac{1}{2r}\right|<\frac{1}{2r}\right\}$.

\begin{prop}\label{P-spectrum-no} Let $\alpha$ satisfy $\alpha_n\uparrow\infty$ and $S_1(\alpha)\not=\emptyset$. Then
\begin{equation}\label{eq.condF}
D(1)\cup\{1\}=\cup_{k\in\N}D(s_0(k))\cup\Sigma\su \sigma(\sC; \Lambda_0(\alpha))
\end{equation}
and also
\begin{equation}\label{eq.condG}
 \sigma(\sC; \Lambda_0(\alpha))\su\cup_{k\in\N}\sigma(\sC; c_0(w_k))\su \ov{D(1)}.
\end{equation}
In particular,
\begin{equation}\label{eq.condH}
\sigma^*(\sC; \Lambda_0(\alpha))=\ov{\sigma(\sC; \Lambda_0(\alpha))}=\ov{D(1)}.
\end{equation}
\end{prop}

\begin{proof} Lemma \ref{L.S}(ii) yields $s_0(\alpha)=1$. By Lemma \ref{L-1},  $\Sigma\su \sigma(\sC; \Lambda_0(\alpha))$.
Fix $k\in\N$ and let $\lambda\in\C\setminus\Sigma$ satisfy $\left|\lambda-\frac{1}{2s_0(k)}\right|<\frac{1}{2s_0(k)}$, i.e., $\lambda\in D(s_0(k))\setminus \Sigma$. For any $y_1\in\C\setminus\{0\}$ define $y\in\C^\N\setminus\{0\}$ by $y_{n+1}:=y_1\prod_{m=1}^n\left(1-\frac{1}{\lambda m}\right)$ for $n\in\N$. It is shown in the proof of Step 1 in the proof of \cite[Proposition 3.7]{ABR-9} that $y\in \ell_1(w_k^{-1})=c_0(w_k)'$ satisfies $\sC_k'y=\lambda y$, where $\sC_k'$ is the dual operator of the Ces\`aro operator $\sC_k\colon c_0(w_k)\to c_0(w_k)$. For $z\in \Lambda_0(\alpha)\su c_0(w_k)$ the vector $(\sC-\lambda I)z$, with $\sC\in \cL(\Lambda_0(\alpha))$, belongs to $c_0(w_k)$. Since $y\in c_0(w_k)'\su \Lambda_0(\alpha)'$, it follows that $\langle (\sC-\lambda I)z, y\rangle=\langle (\sC_k-\lambda I)z, y\rangle=\langle z, (\sC'_k-\lambda I)y\rangle=0$. Therefore $\langle u,y\rangle=0$ for every $u\in \ov{{\rm Im}(\sC-\lambda I)}\su \Lambda_0(\alpha)$ with $y\in\Lambda_0(\alpha)'\setminus\{0\}$. Then $(\sC-\lambda I)\in \cL(\Lambda_0(\alpha))$ cannot be surjective and hence, $\lambda\in \sigma(\sC; \Lambda_0(\alpha))$. This shows that $D(s_0(k))\setminus \Sigma\su \sigma(\sC;\Lambda_0(\alpha))$ and hence, $\Sigma\cup D(s_0(k))\su \sigma(\sC;\Lambda_0(\alpha))$. Since $k\in\N$ is arbitrary and $\frac{1}{2s_0(\alpha)}$ (with $s_0(\alpha)=1$) is the limit of the increasing sequence $\{1/2s_0(k)\}_{k\in\N}$, this establishes \eqref{eq.condF}.

The first containment in \eqref{eq.condG}  follows from \cite[Lemma 2.1]{ABR-1}.  Moreover, since each weight $w_k$, $k\in\N$, is strictly positive and decreasing, Corollary 3.6 of \cite{ABR-9} yields that $\sigma(\sC; c_0(w_k))\su \left\{\lambda\in \C\colon \left|\lambda-\frac{1}{2}\right|\leq \frac{1}{2}\right\}=\ov{D(1)}$ for each $k\in\N$, from which the second containment in \eqref{eq.condG} follows immediately.

The equality $\ov{\sigma(\sC;\Lambda_0(\alpha))}=\ov{D(1)}$ in \eqref{eq.condH} is a consequence of \eqref{eq.condF} and \eqref{eq.condG}.
 So, according to \eqref{eq.condG}, we have $\cup_{k\in\N}\sigma(\sC; c_0(w_k))\su \ov{\sigma(\sC;\Lambda_0(\alpha))}$.
Then \cite[Lemma 2.1]{ABR-1}  implies that $\sigma^*(\sC; \Lambda_0(\alpha))=\ov{\sigma(\sC;\Lambda_0(\alpha))}=\ov{D(1)}$.
\end{proof}

\begin{remark}\label{R.S-C}\rm An examination of the proof of the containments \eqref{eq.condG} shows that the hypothesis $S_1(\alpha)\not=\emptyset$ was not used. Accordingly, for \textit{every} sequence $\alpha$ with $\alpha_n\uparrow\infty$ it is always the case that
\[
\sigma(\sC;\Lambda_0(\alpha))\su \left\{\lambda\in\C\colon \left|\lambda-\frac{1}{2}\right|\leq \frac{1}{2}\right\}.
\]
\end{remark}

The behaviour of the Ces\`aro operator in the Fr\'echet space $H(\D)$ is known; see for example, \cite[pp.65-68]{Dah} and also \cite{AP}. Since $H(\D)$ is isomorphic to $\Lambda_0(\alpha)$ for $\alpha=(n)_{n \in \N}$, the following result also follows from Proposition \ref{p.sprectrum}.

\begin{prop}\label{p.H}
The Ces\`aro operator  $\sC\colon H(\D)\to H(\D)$ satisfies $\sigma(\sC;H(\D))=\sigma_{pt}(\sC;H(\D))=\Sigma$ and $\sigma^*(\sC;H(\D))=\Sigma_0$.
\end{prop}

%%%%%%%%%%%%%%%%%%%%%%%%%%%%%%%%%%%%%%%%%%%%%%%%%%%%%%%%%%%%%%%%%%%%%%%%%%%%%%%%%%%%%%%%%%%%%%%%%%%%%%%%%%%%%%%%%%%%%%%%%%%%%%%%%%%%%%%%%%%%%%%%%%%%%%%%%%%%%%%%%%%%%%%%%%%%%%%%%%%%%%%%%%%%%%%%%%%%%%%%%%%%%%%%%%%%%%%%%%%%%%%%%%%%%%%%%%%%%%%%%%%%%%%%%%%%%%%%%%%%%%%%%%%%%%%%%%%%%%%%%%%%%%%%%%%%%%%%%%%%%%%%%%%%%%%%%%%%%%%%%%%%%%%%%%%%%%%%%%%%%%%%%%%%%%%%%%%%%%%%%%%%%%%%%%%%%%%%%%%%%%%%%%%%%%%%%%%%%%%%%%%%%%%%%%%%%%%%%%%%%%%%%%%%%%%%

\section{Iterates of $\sC$ and  mean ergodicity. }\label{iterates}

An operator  $T \in \cL(X)$,
with $X$ a Fr\'echet space, is \textit{power bounded}  if $\{T^n\}_{n=1}^\infty$ is an equicontinuous subset of $\cL(X)$.   Given $T\in \cL(X)$, the averages
$$
    T_{[n]} := \frac 1 n \sum^n_{m=1} T^m, \qquad      n\in \N,
$$
are called the  Ces\`aro means of $T$.
The operator  $T$ is said to be  \textit{mean ergodic} (resp., \textit{uniformly mean ergodic}) if  $\{T_{[n]}\}^\infty_{n=1}$
is a convergent sequence in $\cL_s (X)$ (resp., in $\cL_b (X)$). A relevant text for mean ergodic operators is \cite{K}.

\begin{prop}\label{p.unif} Let $\alpha$ be any sequence with $\alpha_n\uparrow\infty$. The Ces\`aro operator $\sC\in \cL(\Lambda_0(\alpha))$ is power bounded and uniformly mean ergodic. In particular,
\[
\Lambda_0(\alpha)=\Ker(I-\sC)\oplus\ov{(I-\sC)(\Lambda_0(\alpha))}.
\]
Moreover, it is also the case that   $\Ker (I-\sC)={\rm span}\{\mathbf{1}\}$ and
\[
\ov{(I-\sC)(\Lambda_0(\alpha))}=\{x\in \Lambda_0(\alpha)\colon x_1=0\}=\ov{{\rm span}\{e_n\}_{n\geq 2}}.
\]
\end{prop}

\begin{proof} Clearly   \eqref{e.pb} implies   that  $\sC$ is power bounded, from which  $\tau_b$-$\lim_{n\to\infty}\frac{\sC^n}{n}=0$ follows. Since $\Lambda_0(\alpha)$ is a Montel space, Proposition 2.8 of  \cite{ABR-0} implies that $\sC$ is uniformly mean ergodic.  The proof of the facts that $\Ker (I-\sC)={\rm span}\{\mathbf{1}\}$ and $\ov{(I-\sC)(\Lambda_0(\alpha))}=\{x\in \Lambda_0(\alpha)\colon x_1=0\}=\ov{{\rm span}\{e_n\}_{n\geq 2}}$ follow by applying the same arguments used in the proof of \cite[Proposition 4.1]{ABR-7}.
\end{proof}

For each $m\in\N$, recall the identities
\begin{equation}\label{eq.iter}
(\sC^mx)(n)=\sum_{k=1}^n\left(\begin{array}{c} n-1\\ k-1\end{array}\right)x_k\int_0^1t^{k-1}(1-t)^{n-k}f_m(t)\,dt,\quad n\in\N,
\end{equation}
for all $x\in \C^\N$ , where
\[
f_m(t)=\frac{1}{(m-1)!}\log^{m-1}\left(\frac{1}{t}\right),\quad  t\in (0,1];
\]
see \cite[p.2149]{GF}, \cite[p.125]{L}. The following result is inspired by \cite[Theorem 1]{GF}.

\begin{prop}\label{p.iter} Let $\alpha$ be any sequence with $\alpha_n\uparrow\infty$. The sequence of iterates $\{\sC^m\}_{m\in\N}$ is convergent in $\cL_b(\Lambda_0(\alpha))$.
\end{prop}

\begin{proof} We appeal to  Proposition \ref{p.unif} to show
%we have that $\Lambda_0(\alpha)=\Ker(I-\sC)\oplus \ov{(I-\sC)(\Lambda_0(\alpha))}$,  where $\Ker (I-\sC)={\rm span}\{\mathbf{1}\}$ and $\ov{(I-\sC)(\Lambda_0(\alpha))}=\ov{{\rm span}\{e_n\}_{n\geq 2}}$. We show
that $\{\sC^m\}_{m\in\N}$ converges to the projection onto ${\rm span}\{\mathbf{1}\}$ along $\ov{(I-\sC)(\Lambda_0(\alpha))}$. Indeed, for each $x\in \Lambda_0(\alpha)$, we have that $x=y+z$ with $y\in \Ker (I-\sC)={\rm span}\{\mathbf{1}\}$ and $z\in \ov{(I-\sC)(\Lambda_0(\alpha))}=\ov{{\rm span}\{e_n\}_{n\geq 2}}$. For each $m\in\N$, observe that  $\sC^mx=\sC^m y+\sC^mz$, with $\sC^m y=y\to y$ in $\Lambda_0(\alpha)$ as $m\to\infty$. The claim is that  $\{\sC^mz\}_{m\in\N}$ is also a convergent  sequence in $\Lambda_0(\alpha)$. To this end,  observe that  \eqref{eq.iter} ensures, for each $m\in\N$ and $r\geq 2$,  that $(\sC^me_r)(n)=0$ if $1\leq n<r$ and
\[
(\sC^me_r)(n)=\left(\begin{array}{c} n-1\\ r-1\end{array}\right)\int_0^1t^{r-1}(1-t)^{n-r}f_m(t)\,dt, \quad n\geq r,
\]
where $\{e_r\}_{r\in\N}$ is the canonical basis in $\Lambda_0(\alpha)$. Proceeding as in the proof of \cite[Theorem 1]{GF}, define $g_m(0):=0$, $g_m(t):=tf_m(t)$, for $0<t\leq 1$, and $a_m:=\sup_{t\in [0,1]}g_m(t)$ for $m\in\N$. Then, for each  $r\geq 2$ and $m\in\N$ we obtain  that
$|(\sC^me_r)(n)|\leq \frac{1}{r-1}a_m$ for all $n\in\N$ and hence, for fixed $k\in\N$,  that $
w_k(n)|(\sC^me_r)(n)|\leq \frac{w_k(n)}{r-1}a_m$, for all $n\in\N$. So, for each $r\geq 2$ and $m\in\N$, it follows from \eqref{eq.seminorme} that $p_k(\sC^me_r)\leq \frac{1}{r-1}a_m$. But, $k\in\N$ is arbitrary and  $a_m\to 0$ as $m\to\infty$ (see \cite[Lemma 1]{GF}) and so,  for each $r\geq 2$, we deduce that $\sC^me_r\to 0$ in $\Lambda_0(\alpha)$ as $m\to\infty$. Since the sequence $\{\sC^m\}_{m\in\N}$ is equicontinuous in $\cL(\Lambda_0(\alpha))$ and the linear span of $\{e_r\}_{r\geq 2}$ is dense in $\ov{(I-\sC)(\Lambda_0(\alpha))}$, it follows that $\sC^mz\to 0$ in $\Lambda_0(\alpha)$ as $m\to\infty$ for each $z\in \ov{(I-\sC)(\Lambda_0(\alpha))}$.
So, it has been shown that  $\sC^mx=\sC^my+\sC^mz\to y$ in  $\Lambda_0(\alpha)$ as $m\to\infty$, for each $x\in\Lambda_0(\alpha)$. Since $\Lambda_0(\alpha)$ is a Fr\'echet Montel space, $\{\sC^m\}_{m\in\N}$ is also convergent  in $\cL_b(\Lambda_0(\alpha))$.
\end{proof}

\begin{remark}\label{R.RCon}\rm  Since $\Lambda_0(\alpha)$ is a Fr\' echet Schwartz space, the sequence  $\{\sC^m\}_{m\in\N}$ is even \textit{rapidly convergent} in $\cL_b(\Lambda_0(\alpha))$, in the sense of \cite{ABR-St}.
\end{remark}

It is known that there exist power bounded, uniformly mean ergodic operators $S$ acting on certain K\"othe echelon spaces for which  the range of $(I-S)$ is \textit{not} a closed subspace, \cite[Propositions 3.1 and 3.3]{ABR-7}. This shows that a result of  Lin \cite{Li}, valid in the Banach space setting,  cannot be extended to general Fr\'echet spaces.
However, we do have the following result.

\begin{prop}\label{p.range} Let $\Lambda_0(\alpha)$ be nuclear. Then the  range $(I-\sC)^m(\Lambda_0(\alpha))$ is a closed subspace of  $\Lambda_0(\alpha)$ for each $m\in\N$.
\end{prop}

\begin{proof} Consider first  $m=1$.
By Proposition \ref{p.unif} we have that $\ov{(I-\sC)(\Lambda_0(\alpha))}=\{x\in \Lambda_0(\alpha)\colon x_1=0\}$. Set $X_1(\alpha):=\{x\in \Lambda_0(\alpha)\colon x_1=0\}\su \Lambda_0(\alpha)$. Clearly $(I-\sC)(\Lambda_0(\alpha))\su X_1(\alpha)$.    The claim is that  $(I-\sC)(X_1(\alpha))=(I-\sC)(\Lambda_0(\alpha))$. One inclusion is obvious. To estabish the other inclusion, observe that
\begin{equation}\label{eq.rango1}
(I-\sC)x=\left(0, x_2-\frac{x_1+x_2}{2},x_3-\frac{x_1+x_2+x_3}{3},\ldots\right),\quad x=(x_n)_{n\in\N}\in \Lambda_0(\alpha)
\end{equation}
and, in particular,  that
\begin{equation}\label{eq.rango2}
(I-\sC)y=\left(0, \frac{y_2}{2},y_3-\frac{y_2+y_3}{3},y_4-\frac{y_2+y_3+y_4}{4},\ldots\right),\quad y=(y_n)_{n\in\N}\in X_1(\alpha)
\end{equation}
Fix $x\in \Lambda_0(\alpha)$. It follows from \eqref{eq.rango1} that
\begin{equation}\label{eq.rango3}
x_j-\frac{1}{j}\sum_{k=1}^jx_k=\frac{1}{j}\left((j-1)x_j-\sum_{k=1}^{j-1}x_k\right),\quad j\geq 2,
\end{equation}
is the $j$-th coordinate of the vector $(I-\sC)x$. Set $y_i:=x_i-x_1$ for all $i\in\N$ and observe that the vector $y:=(y_i)_{i\in\N}\in X_1(\alpha)$ because
 $(0,1,1,1,\ldots)\in \Lambda_0(\alpha)$; see Lemma \ref{L-1}. Now, via \eqref{eq.rango2}, one shows that the $j$-th coordinate of $(I-\sC)y$ is given by \eqref{eq.rango3} for $j\geq 2$. So, $(I-\sC)x=(I-\sC)y\in (I-\sC)(X_1(\alpha))$.

To show that the range $(I-\sC)(\Lambda_0(\alpha))$ is closed in $\Lambda_0(\alpha)$ it suffices to show that the (restricted) continuous linear operator  $(I-\sC)|_{X_1(\alpha)}\colon X_1(\alpha)\to X_1(\alpha)$ is bijective, actually surjective (as it is clearly injective by \eqref{eq.rango2}).
To establish surjectivity,   observe that $X_1(\alpha)=\cap_{k\in\N}X^{(k)}$, with $X^{(k)}:=\{x\in c_0(w_k)\colon x_1=0\}$ being a closed subspace of $c_0(w_k)$ for all $k\in\N$.
 Actually,  set $\tilde{\alpha}=(\alpha_{n+1})_{n\in\N}$ and $\tilde{w}_k(n):=w_k(n+1)$ for all $k,\ n\in\N$. Then $X_1(\alpha)$ is topologically  isomorphic to $\Lambda_0(\tilde{\alpha}):=\cap_{k\in\N} c_0(\tilde{w}_k)$ via the left shift operator $S\colon X_1(\alpha)\to \Lambda_0(\tilde{\alpha})$ given by $S(x):=(x_2,x_3,\ldots)$ for $x=(x_n)_{n\in\N}\in X_1(\alpha)$. The claim is that  the operator $A:=S\circ (I-\sC)|_{X_1(\alpha)}\circ S^{-1}\in \cL(\Lambda_0(\tilde{\alpha}))$  is bijective with $A^{-1}\in \cL(\Lambda_0(\tilde{\alpha}))$.

To verify  this claim observe that, when considered as acting in  $\C^\N$, the operator $A\colon \C^\N\to \C^\N$ is bijective and its inverse $B:=A^{-1}\colon \C^\N\to \C^\N$ is determined by the lower triangular matrix $B=(b_{nm})_{n,m}$ with entries given by: for each $n\in\N$ we have $b_{nm}=0$ if $m>n$, $b_{nm}=\frac{n+1}{n}$ if $m=n$ and $b_{nm}=\frac{1}{m}$ if $1\leq m<n$. To show   that $B$ is also the inverse of $A$ when acting on $\Lambda_0(\tilde{\alpha})$, we only need to verify  that   $B\in \cL(\Lambda_0(\tilde{\alpha}))$. To establish  this  it suffices to prove,  for each $k\in\N$, that there exists $l\geq k$ such that $\Phi^{-1}_{\tilde{w}_k}\circ B\circ \Phi_{\tilde{w}_l}\in \cL(c_0)$ where, for each $h\in\N$, the operator $\Phi_{\tilde{w}_h}\colon c_0(\tilde{w_h})\to c_0$  given by $ \Phi_{\tilde{w}_h}(x)=(\tilde{w}_h(n+1)x_n)_n$ for $x\in c_0(\tilde{w}_h)$ is a surjective isometry. Whenever $k\in\N$ and  $l>k$, the lower triangular matrix corresponding to $\Phi^{-1}_{\tilde{w}_k}\circ B\circ \Phi_{\tilde{w}_l}$ is given by $D_{l,k}:=(\frac{w_k(n+1)}{w_l(m+1)}b_{nm})_{n,m}$. For each fixed $m\in\N$, note that $D_{l,k}$ satisfies
\[
\lim_{n\to\infty}\frac{w_k(n+1)}{w_l(m+1)}b_{nm}=\frac{1}{mw_l(m+1)}\lim_{n\to\infty}w_k(n+1)=0.
\]
Moreover, for $n\in\N$ fixed, we also have
\begin{eqnarray*}
& &\sum_{m=1}^\infty\frac{w_k(n+1)}{w_l(m+1)}b_{nm}=\frac{(n+1)}{n}\frac{w_k(n+1)}{w_l(n+1)}+w_k(n+1)\sum_{m=1}^{n-1}\frac{1}{mw_l(m+1)}\\
& &\qquad \leq 2+e^{-\frac{\alpha_{n+1}}{k}}\sum_{m=1}^{n-1}\frac{e^{\frac{\alpha_{m+1}}{l}}}{m}\leq 2+e^{\alpha_{n+1}(\frac{1}{l}-\frac{1}{k})}\sum_{m=1}^{n-1}\frac{1}{m}\\
& &\qquad \leq 2+e^{\alpha_{n+1}(\frac{1}{l}-\frac{1}{k})}(1+\log n).
\end{eqnarray*}
These inequalities are valid for every $n\in\N$, whenever $k\in\N$ and $l>k$.

Fix $k\in\N$. Since $\Lambda_0(\alpha)$ is nuclear, there is $l>k$ such that the quantity $\sup_{n\in\N}(\log (n)- (\frac{1}{k}-\frac{1}{l})\alpha_n)=:M<\infty$ (see Proposition \ref{P-K2}). So, for every $n\in\N$, we have
\[
\log n\leq M+(\frac{1}{k}-\frac{1}{l})\alpha_n\leq M+(\frac{1}{k}-\frac{1}{l})\alpha_{n+1}\leq M+e^{(\frac{1}{k}-\frac{1}{l})\alpha_{n+1}}.
\]
This implies that $e^{\alpha_{n+1}(\frac{1}{l}-\frac{1}{k})}\log (n)\leq 1+Me^{\alpha_{n+1}(\frac{1}{l}-\frac{1}{k})}\leq 1+M$ for all $n\in\N$. It then follows
that
\[
\sup_{n\in\N}e^{\alpha_{n+1}(\frac{1}{l}-\frac{1}{k})}(1+\log n)<\infty.
\]
 Thus,  by \cite[Theorem 4.51-C]{T},  $\Phi^{-1}_{\tilde{w}_k}\circ B\circ \Phi_{\tilde{w}_l}\in \cL(c_0)$, as required. That is, $(I-\sC)(\Lambda_0(\alpha))$ is closed in $\Lambda_0(\alpha)$.

Since $(I-\sC)(\Lambda_0(\alpha))$ is closed, it follows from Proposition \ref{p.unif} that $\Lambda_0(\alpha)=\Ker(I-\sC)\oplus (I-\sC)(\Lambda_0(\alpha))$.
The proof of (2)$\Rightarrow$(5) in Remark 3.6 of \cite{ABR-7} shows that $(I-\sC)^m(\Lambda_0(\alpha))$ is closed for all $m\in\N$.
\end{proof}

A Fr\'echet space operator $T\in \cL(X)$, with $X$ separable, is called \textit{hypercyclic} if there exists $x\in X$ such that the orbit $\{T^nx\colon n\in\N_0\}$ is dense in $X$. If, for some $z\in X$ the projective orbit $\{\lambda T^n z\colon \lambda\in\C,\ n\in\N_0 \}$ is dense in $X$, then $T$ is called \textit{supercyclic}. Clearly, hypercyclicity  implies supercyclicity.

\begin{prop}\label{p.Sup} Let $\alpha$ be any sequence with $\alpha_n\uparrow\infty$. Then $\sC\in \cL(\Lambda_0(\alpha))$ is not supercyclic and hence, also not hypercyclic.
\end{prop}

\begin{proof} Suppose that $\sC$ is supercyclic. Since the canonical inclusion $\Lambda_0(\alpha)\su c_0(w_1)$ is continuous and $\Lambda_0(\alpha)$   is dense in $c_0(w_1)$, it follows that $\sC_1\in \cL(c_0(w_1))$ is supercyclic in $c_0(w_1)$ which contradicts Proposition 4.13 of \cite{ABR-9}.
\end{proof}

\bigskip

\textbf{Acknowledgements.} The authors thank the referee for some useful suggestions and improvements.

The research of the first two authors was partially supported by the projects  MTM2013-43540-P, MTM2016-76647-P and GVA Prometeo II/2013/013.

\bigskip

\bibliographystyle{plain}

\begin{thebibliography}{99}
\bibitem{A-B} A.M. Akhmedov, F. Ba\c{s}ar, \textit{On the fine spectrum of the  Ces\`aro operator in $c_0$.} Math. J. Ibaraki Univ. \textbf{36} (2004), 25--32.
\bibitem{A-B2} A.M. Akhmedov, F. Ba\c{s}ar, \textit{The fine spectrum of the  Ces\`aro operator $C_1$ over the sequence space $bv_p$, ($1\leq p< \infty$).} Math. J. Okayama Univ. \textbf{50} (2008), 135--147.
\bibitem{ABR-0} A.A. Albanese, J. Bonet, W.J. Ricker, \textit{Mean ergodic operators in Fr\'echet spaces.} Ann. Acad. Sci. Fenn. Math. \textbf{34} (2009), 401--436.
\bibitem{ABR-2} A.A. Albanese, J. Bonet, W.J. Ricker, \textit{Mean ergodic semigroups of  operators.} Rev. R. Acad. Cien. Serie A, Mat., RACSAM, \textbf{106} (2012), 299--319.
\bibitem{ABR-5} A.A. Albanese, J. Bonet, W.J. Ricker, \textit{Montel resolvents and uniformly mean ergodic semigroups of linear operators.} Quaest. Math. \textbf{36} (2013), 253--290.
\bibitem{ABR-7} A.A. Albanese, J. Bonet, W.J. Ricker, \textit{Convergence of arithmetic means of operators in Fr\'echet spaces.}  J. Math. Anal. Appl. \textbf{401} (2013), 160--173.
\bibitem{ABR-St} A.A. Albanese, J. Bonet, W.J. Ricker, \textit{Characterizing Fr\'echet-Schwartz spaces via power bounded operators.} Studia Math. \textbf{224} (2014), 25--45.
\bibitem{ABR} A.A. Albanese, J. Bonet, W.J. Ricker, \textit{Spectrum and compactness of the  Ces\`aro operator on weighted $\ell_p$ spaces.} J. Aust. Math. Soc., \textbf{99} (2015), 287--314.
\bibitem{ABR-9} A.A. Albanese, J. Bonet, W.J. Ricker, \textit{Mean ergodicity and spectrum of the Ces\`aro operator on weighted $c_0$ spaces.} Positivity, \textbf{20} (2016), 761--803.
\bibitem{ABR-1} A.A. Albanese, J. Bonet, W.J. Ricker, \textit{The Ces\`aro operator in the Fr\'echet spaces $\ell^{p+}$ and $L^{p-}$.} Glasgow Math. J. \textbf{59} (2017), 273-287.
\bibitem{AP} A. Aleman, A.-M. Persson, \textit{Resolvent estimates and decomposable extensions of generalized Ces\`aro operators.} J. Funct. Anal. \textbf{258} (2010), 67--98.
\bibitem{BHS} A. Brown, P.R. Halmos, A.L. Shields, \textit{Ces\`aro operators.} Acta Sci. Math. Szeged \textbf{26} (1965), 125--137.
\bibitem{C-R} G.P. Curbera, W.J. Ricker, \textit{Solid extensions of the Ces\`aro operator on $\ell^p$ and $c_0$.} Integr. Equ. Oper. Theory \textbf{80} (2014), 61--77.
\bibitem{C-R1} G.P. Curbera, W.J. Ricker, \textit{The Ces\`aro operator and unconditional Taylor series in Hardy spaces.} Integr. Equ. Oper. Theory \textbf{83} (2015), 179--195.
\bibitem{Dah} A. Dahlner, \textit{Some estimates in harmonic analysis}, Doctoral Thesis in Mathematical Sciences 2003:4, Center of Mathematical Sciences, Lund University 2003.
\bibitem{GS} P. Galanopoulos, A.G. Siskakis, \textit{Hausdorff matrices and composition operators}, Illinois J. Math. \textbf{45} (2001), 757--773.
\bibitem{GF} F. Galaz Fontes, F.J. Solis, \textit{Iterating the Ces\`aro operators.} Proc. Amer. Math. Soc. \textbf{136} (2008), 2147--2153.
\bibitem{H43} G.H. Hardy, \textit{An inequality for Hausdorff means}, J. London Math. Soc. \textbf{18} (1943), 46--50.
\bibitem{H49} G.H. Hardy, \textit{Divergent Series.} Clarendon Press. Oxford, 1949.
\bibitem{Hau} F. Hausdorff, \textit{Summationsmethoden und Momentfolgen I}, Math. Z. \textbf{9} (1921), 74--109.
\bibitem{K} U. Krengel,  \textit{Ergodic Theorems.} de Gruyter Studies in Mathematics, \textbf{6}. Walter de Gruyter Co., Berlin, 1985.
\bibitem{L} G. Leibowitz, \textit{Spectra of discrete Ces\`aro operators.} Tamkang J. Math. \textbf{3} (1972), 123--132.
\bibitem{Li} M. Lin, \textit{On the uniform ergodic theorem.} Proc. Amer. Math. Soc. \textbf{43} (1974), 337--340.
\bibitem{MV} R. Meise, D. Vogt, \textit{Introduction to Functional Analysis.} Clarendon Press, Oxford, 1997.
\bibitem{M} M. Mure\c{s}an, \textit{A Concrete Approach to Classical Analysis.} Springer, Berlin, 2008.
\bibitem{Ok} J.I. Okutoyi, \textit{On the spectrum of $C_1$ as an operator on $bv_0$.} J. Austral. Math. Soc. (Series A) \textbf{48} (1990), 79--86.
\bibitem{R} J.B. Reade, \textit{On the spectrum of the Ces\`aro operator.} Bull. London Math. Soc. \textbf{17} (1985), 263--267.
\bibitem{Rud} O. Rudolf, \textit{Hausdorff-Operatoren auf BK-Ra\"umen und Halbgruppen linearen Operatoren}, Mitt. Math. Sem. Giessen \textbf{241} (2000), pp. 1--100.
\bibitem{T} A.E. Taylor, \textit{Introduction to Functional  Analysis.} John Wiley \& Sons, New York 1958.
\bibitem{V1} D. Vogt, \textit{Charakterisierung der Unterr\"aume eines nuklearen stabilen Potenzreihenraumes von endlichen Typ.} Studia Math. \textbf{71} (1982), 251--270.
\bibitem{V2} D. Vogt, \textit{Eine Charakterisierung der Potenzreihenr\"aume von endlichen Typ und ihre Folgerungen.} Manuscripta Math. \textbf{37} (1982), 269--301.
\bibitem{V3} D. Vogt, \textit{Fr\'echetr\"aume, zwischen denen jede stetige lineare Abbildung beschr\"ankt ist.} J. reine angew. Math. \textbf{345} (1983), 182--200.
\end{thebibliography}

\end{document}